\documentclass{amsart}
\usepackage{amssymb,mathrsfs}

\def\bC {\mathbf{C}}

\def\bH {\mathbf{H}}

\def\bR {\mathbf{R}}

\def\fH {\mathfrak{H}}
\def\fS {\mathfrak{S}}

\def\cC {\mathcal{C}}
\def\cD {\mathcal{D}}
\def\cE {\mathcal{E}}
\def\cF {\mathcal{F}}

\def\cH {\mathcal{H}}

\def\cL {\mathcal{L}}

\def\cS {\mathcal{S}}

\def\cV {\mathcal{V}}
\def\cW {\mathcal{W}}

\def\g {{\gamma}}
\def\Ga {{\Gamma}}
\def\de {{\delta}}
\def\eps {{\epsilon}}

\def\L {{\Lambda}}
\def\si {{\sigma}}

\def\Om {{\Omega}}

\def\d {{\partial}}
\def\grad {{\nabla}}
\def\Dlt {{\Delta}}

\def\rstr {{\big |}}

\def\la {\langle}
\def\ra {\rangle}

\newcommand{\Span}{\operatorname{span}}

\newcommand{\Tr}{\operatorname{trace}}

\newcommand{\Lip}{\operatorname{Lip}}

\newcommand{\MKd}{\operatorname{dist_{MK,2}}}
\newcommand{\Op}{\operatorname{OP}}

\newcommand{\ba}{\begin{aligned}}
\newcommand{\ea}{\end{aligned}}

\newcommand{\be}{\begin{equation}}
\newcommand{\ee}{\end{equation}}

\newcommand{\lb}{\label}

\newtheorem{Thm}{Theorem}[section]
\newtheorem{Rmk}[Thm]{Remark}

\newtheorem{Lem}[Thm]{Lemma}
\newtheorem{Def}[Thm]{Definition}

\begin{document}

\title[Schr\"odinger in Mean-Field and Semiclassical Regime]{The Schr\"odinger Equation in the Mean-Field\\ and Semiclassical Regime}

\author[F. Golse]{Fran\c cois Golse}
\address[F.G.]{Ecole polytechnique, CMLS, 91128 Palaiseau Cedex, France}
\email{francois.golse@polytechnique.edu}

\author[T. Paul]{Thierry Paul}
\address[T.P.]{CNRS \& Ecole polytechnique, CMLS, 91128 Palaiseau Cedex, France}
\email{thierry.paul@polytechnique.edu}

\begin{abstract}
In this paper, we establish (1) the classical limit of the Hartree equation leading to the Vlasov equation, (2) the classical limit of the $N$-body linear Schr\"odinger equation uniformly in $N$ leading to the $N$-body Liouville equation of classical 
mechanics and (3) the simultaneous mean-field and classical limit of the $N$-body linear Schr\"odinger equation leading to the Vlasov equation. In all these limits, we assume that the gradient of the interaction potential is Lipschitz continuous.
All our results are formulated as estimates involving a quantum analogue of the Monge-Kantorovich distance of exponent $2$ adapted to the classical limit, reminiscent of, but different from the one defined in [F. Golse, C. Mouhot, T. Paul, 
Commun. Math. Phys. \textbf{343} (2016), 165--205]. As a by-product, we also provide bounds on the quadratic Monge-Kantorovich distances between the classical densities and  the Husimi functions of the quantum density matrices.
\end{abstract}



\date{\today}

\maketitle

 

\section{Introduction}\lb{S-Intro}


Consider the nonrelativistic quantum dynamics of a system of $N$ particles interacting through a two-body potential with a mean-field type coupling constant. (In other words, the coupling constant is chosen so that the kinetic and potential energies 
of typical $N$-particle configurations are of the same order of magnitude.) After suitable rescalings of the various quantities involved (see the introduction of \cite{FGMouPaul} for details), the dynamics happens to be governed by a two-parameter 
family of  Schr\"odinger equations indexed by $\hbar$ and $N$, of the form
\be\lb{CPSchr}
\left\{
\ba
{}&i\hbar\d_t\Psi_{N,\hbar}=-\tfrac12\hbar^2\sum_{k=1}^N\Dlt_{x_k}\Psi_{N,\hbar}+\frac1{2N}\sum_{k,l=1}^NV(x_k-x_l)\Psi_{N,\hbar}\,,
\\
&\Psi_{N,\hbar}\rstr_{t=0}=\Psi_{N,\hbar}^{in}\,.
\ea
\right.
\ee
Here $\Psi_{N,\hbar}^{in}\equiv\Psi_{N,\hbar}^{in}(x_1,\ldots,x_N)\in\bC$ and $\Psi_{N,\hbar}\equiv\Psi_{N,\hbar}(t,x_1,\ldots,x_N)\in\bC$ are the wave functions of the $N$-particle system initially and at time $t$ respectively, while $V$ is the
real-valued, rescaled interaction potential.

We are concerned with various asymptotic limits of this dynamics as $N\to\infty$ and $\hbar\to 0$, which are represented in the diagram below.

In \cite{FGMouPaul}, the limits corresponding to the horizontal arrows have been studied in detail. Following an idea of Dobrushin \cite{Dobru}, we arrived at quantitative estimates on distances metrizing the weak convergence of probability
measures, or of their quantum analogues. However, the method used in \cite{FGMouPaul} differs from Dobrushin's in \cite{Dobru} in several ways. First, we avoid formulating our results in terms of the $N$-particle phase-space empirical measure,
which does not seem to have any clear quantum analogue. Likewise, we avoid expressing the solution of the $N$-particle Liouville equation by the method of characteristics, of which there is no convenient quantum analogue. Our approach 
of the Vlasov limit of the $N$-particle Liouville equation is of a completely Eulerian nature, so that its extension to the quantum dynamics is rather straightforward.

\begin{center}

\bigskip
\begin{tabular}{ccc}
{\fbox{\bf Schr\"odinger}}& {$\stackrel{N\to\infty}{\longrightarrow}$}& {\fbox{\bf Hartree}} 
\\ [7mm]
\scriptsize{${\hbar\to 0}$}\ \LARGE{$\downarrow$}\ \ \ \ \ \ \  &\LARGE$\searrow$ & \ \ \ \ \ \ \  \LARGE{{$\mathbf{\downarrow}$}}\ 
\scriptsize{${\hbar\to 0}$} 
\\ [7mm]
{\fbox{\bf Liouville}}& {$\stackrel{N\to\infty}{\longrightarrow}$}&{\fbox {\bf Vlasov}} 
\end{tabular}
\bigskip

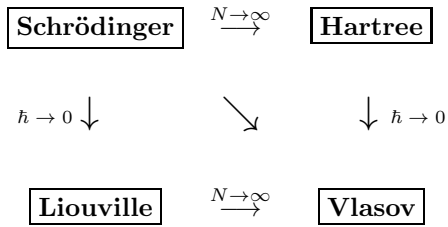
\centerline{Diagram 1: The large $N$ and small $\hbar$ asymptotic limits}

\bigskip
\bigskip
\end{center}

The main result in \cite{FGMouPaul} is that the upper horizontal arrow, i.e. the large $N$, mean-field limit of the $N$-particle linear Schr\"odinger equation to the Hartree equation is uniform as $\hbar\to 0$. This uniform convergence was stated
in terms of a nonnegative quadratic quantity $MK_2^\hbar(R,R')$, defined for pairs of density operators $R,R'$ (nonnegative operators of trace $1$), obtained by quantization of the quadratic Monge-Kantorovich distance $\MKd$ (see
section \ref{S-MRes} where the definition of this distance is recalled for the reader's convenience). The corresponding estimate was proved under the assumption that the interaction force field $\grad V$ is bounded and Lipschitz continuous.

An important observation is that $MK_2^\hbar$ is \textit{not} a distance on the set of density operators. However, $MK_2^\hbar(R,R')$ remains to within $O(\hbar)$ of the quadratic Monge-Kantorovich distance between classical objects
attached to $R,R'$. Specifically, 
$$
\MKd(\tilde W_\hbar[R],\tilde W_\hbar[R'])^2-O(\hbar)\le MK_2^\hbar(R,R')^2
$$
where $\tilde W_\hbar$ designates the Husimi transform, while
$$
MK_2^\hbar(R,R')^2\le\MKd(\mu,\mu')^2+O(\hbar)
$$
if $R,R'$ are T\"oplitz operators at scale $\hbar$ whose respective symbols are (up to some normalization) the Borel probability measures $\mu,\mu'$. (See section \ref{S-MRes}, where the definitions of the Husimi transform and of T\"oplitz 
operators are recalled.)

\smallskip
Proving the uniformity as $\hbar\to 0$ of the mean-field limit corresponding to the horizontal arrow in Diagram 1 is obviously the key to obtaining the convergence of the $N$-body Schr\"odinger equation to the Vlasov equation in the joint limit
as $N\to\infty$ and $\hbar\to 0$. This convergence holds independently of distinguished scaling assumptions that would link $N$ and $\hbar$, as a consequence of the main result in \cite{FGMouPaul}, and of Theorem IV.2 in \cite{LionsPaul},
which establishes the validity of the classical limit of the Hartree equation, leading to the Vlasov equation (the right vertical arrow in Diagram 1).

Indeed, the classical limit of the Hartree equation leading to the Vlasov equation has been first investigated mathematically in \cite{LionsPaul} for very general initial data and potentials, including the Coulomb potential, in terms of an 
appropriate weak topology and related compactness methods. The price to pay for this generality on the data is the lack of quantitative information on the rate of convergence. Besides, convergence is proved along sequences $\hbar_n\to 0$. 
Under additional assumptions on the initial data and the interaction potential, one can prove \cite{APPP1} that the Wigner function of the solution of Hartree's equation is $L^2$-close to its weak limit, i.e. to the solution of the Vlasov equation.
The quantum counterpart of this result is the closeness of the solution of Hartree's equation to the (Weyl) quantization of the solution of Vlasov in Hilbert-Schmidt norm. The convergence rate is $O(\hbar^\alpha)$ with $\alpha<1$ depending 
on the  initial data and potential (assumed regular enough). The case $\alpha=1$ was recently treated in \cite{BPSS}, where a more exhaustive bibliography on the subject can be found. More precisely, the reference \cite{BPSS} provides an
estimate of the difference between the solution of the Hartree equation and the Weyl quantization of the solution of the Vlasov equation in trace norm. Unlike Hilbert-Schmidt norm estimates, trace norm estimates do not have classical analogues. 
At variance with these two kinds of estimates (with either the trace norm or the Hilbert-Schmidt norm), our ``pseudo-distance'' $E_\hbar$ defined below and used in Theorem \ref{T-HV} connects directly a quantum object (a solution of the Hartree 
equation) with a classical one (a solution of the Vlasov equation). Besides, $E_\hbar$ has a (kind of) classical analogue: Theorem 2.5 provides an upper bound for the classical quadratic Monge-Kantorovich distance between the solution of 
the Vlasov equation and the Husimi function of the solution of the Hartree equation, up to terms of order $O(\hbar)$ which vanish in the semiclassical limit. The case of pure states is treated in \cite{APPP2} for initial data given by coherent states, 
and the solution of Hartree's equation is computed at leading order in terms of the linearization of the ``Vlasov flow". Most likely, the method used in \cite{APPP2} can be extended to any order in the expansion of the Hartree solution in 
powers of $\hbar^{1/2}$. 

Our first main result in the present paper, Theorem \ref{T-HV}, bears on a \textit{quantitative estimate} for the classical limit of Hartree's equation leading to the Vlasov equation, i.e. on the convergence rate for Theorem IV.2 in \cite{LionsPaul}. It involves  a hybrid quantity $E_\hbar$ built on $\MKd$ and $MK^\hbar_2$, whose main properties are stated in Theorem \ref{T-PtyE} and which allows comparing classical and quantum objects (see Definition \ref{ehbar} below). This convergence
rate is established under the assumption that the interaction force field $\grad V$ is bounded and Lipschitz continuous, and involves the Lipschitz constant of $\grad V$. Observe that the large-time growth of the convergence rate obtained in
Theorem \ref{T-HV} is exponential, at variance with the estimates obtained in \cite{APPP1,BPSS}, which are super-exponential. Our result can be also formulated as a direct comparison between the Vlasov solution and the Husimi function of 
the Hartree solution. (A similar comparison can be found in our previous work with C. Mouhot, at least implicitly, as a consequence of Theorem 2.4 and 2.3 (2) in \cite{FGMouPaul}.) Since a density operator is completely determined by its Husimi 
function for each $\hbar>0$ (see Remark \ref{husforever} below), the estimate in Theorem \ref{T-HV}  involves all the information included in the quantum density, i.e. the Hartree solution. We also recall that the Husimi and the Wigner functions
of a density operator have the same classical limit (i.e. the same limit as $\hbar\to 0$): see Theorem III.1 (1) in \cite{LionsPaul}.

Our second main result, Theorem \ref{T-NSV}, establishes the limit of the $N$-body linear Schr\"odinger equation leading to the Vlasov equation in the limit as $N\to\infty$ and $\hbar\to 0$ jointly, again with a convergence rate where the
``distance'' between the single-particle marginal of the $N$-body density operator and the Vlasov solution is expressed in terms of the quantity $E_\hbar$ defined in Definition \ref{ehbar}. Since the quantities $MK_2^\hbar$ and $E_\hbar$ 
are not distances, the estimate in Theorem \ref{T-NSV} does not immediately follow from Theorem \ref{T-HV} and the uniform convergence result, i.e. Theorem 2.4 in \cite{FGMouPaul}. For that reason, we have provided a direct proof of 
the  limit as $N\to\infty$ and $\hbar\to 0$ jointly, corresponding to the  diagonal arrow in Diagram 1. This proof combines ideas from the proofs of Theorem \ref{T-HV} and of Theorem 2.4 in \cite{FGMouPaul}. As for Theorem \ref{T-HV},
the convergence rate obtained in Theorem \ref{T-NSV} grows exponentially fast as $t\to+\infty$. The convergence proof involves a stability and a consistency estimate, in the sense of the Lax equivalence theorem in numerical analysis
\cite{LaxRicht}. Perhaps the new element in this proof, in addition to the idea of measuring the ``distance'' between a classical and a quantum density by  the $E_\hbar$ functional, is the following observation: while the consistency estimate 
involves the $N$-fold tensor product of copies of the Vlasov solution, the stability estimate is most conveniently formulated in terms of \textit{the first equation only} in the BBGKY hierarchy of the quantum $N$-body problem.

To the best of our knowledge, the first work concerning this subject is the seminal paper by Graffi-Martinez-Pulvirenti \cite{GMP}. For each sequence $\hbar(N)\to 0$ as $N\to\infty$ and each monokinetic solution of the Vlasov equation 
(i.e. of the form $f(t,x,\xi):=\rho(t,x)\delta(\xi-u(t,x))$), the Wigner transform at scale $\hbar(N)$ of the first marginal of the solution of the Schr\"odinger equation is proved to converge to the solution of the Vlasov equation over the time interval 
$[0,T]$. A priori, the convergence rate and the time $T$ depend both on the Vlasov solution $f$ and on the sequence $\hbar(N)$ (see Theorem 1.1 in \cite{GMP}). On the preceding diagram, the result proved in \cite{GMP} corresponds 
to the left vertical and bottom horizontal arrows along distinguished sequences $(\hbar(N);N)$, over time intervals which may involve the dependence of $\hbar$ in terms of $N$.

Another approach of the same problem can be found in \cite{PP}: it is proved that each term in the semiclassical expansion as $\hbar\to 0$ of the quantum $N$-body problem converges as $N\to\infty$ to the corresponding term in the 
semiclassical expansion of Hartree's equation.  

One should also mention the earlier reference \cite{NarnhoSewell}, which treated the mean-field limit for systems of $N$ fermions with $\hbar(N)=N^{-1/3}$ (see also \cite{Spohn}, together with the more recent reference \cite{BPSS}).

Finally, we have applied the ideas in the proofs of Theorems \ref{T-HV} and \ref{T-NSV}, i.e. using the $E_\hbar$ functional and the first equation in the BBGKY hierarchy of the quantum $N$-body problem in the stability estimate, to the 
classical limit of the quantum $N$-body problem (i.e. the $N$-body Schr\"odinger equation) to the classical $N$-body problem (i.e. the $N$-body Liouville equation) in the limit as $\hbar\to 0$. This is the left vertical arrow in Diagram 1. 
This limit has been the subject matter of a large body of mathematical literature for $N$ fixed. At variance with all these results, the main novelty in Theorem \ref{T-SL} is that this vanishing $\hbar$-limit is uniform as $N\to\infty$. Here again,
a convergence rate in terms of the $E_\hbar$ functional and of the Lipschitz constant of the interaction force $\grad V$ is given.

\smallskip 
The precise statements of all the assumptions and main results  of this article, together with the general setting of our approach are given in detail in the next section. 

\newpage

\section{Main Results}\lb{S-MRes}


\subsection{From the Hartree Equation to the Vlasov Equation}\label{fhtv}


Let us consider the Cauchy problem for the Hartree equation, written in terms of density operators:
\be\lb{Hartree}
i\hbar\d_tR_\hbar=[-\tfrac12\hbar^2\Dlt_x+V\star_x\rho[R_\hbar](t,x),R_\hbar]\,,\quad R_\hbar\rstr_{t=0}=R_\hbar^{in}\,.
\ee
The notation is as follows: $R_\hbar\equiv R_\hbar(t)\in\cL(\fH)$ is the time-dependent density operator on the state space $\fH:=L^2(\bR^d)$. Henceforth, it is assumed that
$$
R_\hbar(t)=R_\hbar(t)^*\ge 0\,,\qquad\hbox{ and that }\Tr(R_\hbar(t))=1\quad\hbox{ for all }t\ge 0\,.
$$
The set of all bounded operators on $\fH$ satisfying these properties is denoted by $\cD(\fH)$. Denoting by $r_\hbar(t,x,y)$ the integral kernel of $R_\hbar(t)$, i.e.
$$
R_\hbar(t)\phi(x):=\int_{\bR^d}r_\hbar(t,x,y)\phi(y)dy\,,
$$
we set\footnote{If $R\in\cD(\fH)$, then both $R$ and $R^{1/2}$ are Hilbert-Schmidt operators on $L^2(\bR^d)$ and therefore have integral kernels denoted respectively $r\equiv r(x,y)$ and $r_{1/2}\equiv r_{1/2}(x,y)$ in 
$L^2(\bR^d\times\bR^d)$. Since $R^{1/2}$ is self-adjoint
$$
r(x,y)=\int_{\bR^d}r_{1/2}(x,z)\overline{r_{1/2}(y,z)}dz\hbox{ for a.e. }x,y\in\bR^d\,.
$$
By the Fubini theorem, the function 
$$
\rho[R]:\,x\mapsto r(x,x):=\int_{\bR^d}|r_{1/2}(x,z)|^2dz\hbox{ belongs to }L^1(\bR^d)\,,
$$
(see Example 1.18 in chapter X, \S 1 in \cite{Kato}) and, by the Cauchy-Schwarz inequality
$$
\iint_{\bR^d\!\times\!\bR^d}\!|r(x\!+\!h,x)\!-\!r(x,x)|^2dxdy\!\le\!\|r_{1/2}\|_{L^2(\bR^d\!\times\!\bR^d)}\iint_{\bR^d\!\times\!\bR^d}\!|r_{1/2}(x\!+\!h,y)\!-\!r(x,y)|^2dxdy\!\to\! 0
$$
as $|h|\to 0$. In other words, the function $(x,h)\mapsto r(x+h,x)$ belongs to $C(\bR^d_h;L^1(\bR^d))$. This is a special case of Lemma 2.1 (1) in \cite{BGM}.}
$$
\rho[R_\hbar](t,x):=r_\hbar(t,x,x)\,.
$$
The interaction potential $V$ is assumed to be an even, real-valued, bounded function in the class $C^{1,1}(\bR)$. The existence and uniqueness theory for the Cauchy problem (\ref{Hartree}) has been studied in \cite{BoveDPF}.

On the other hand, consider the Vlasov equation
\be\lb{Vlasov}
\d_tf+\{\tfrac12|\xi|^2+V\star_x\rho_f(t,x),f\}=0\,,\quad f\rstr_{t=0}=f^{in}\,.
\ee
Here $f\equiv f(t,x,\xi)$  is a time-dependent probability density on $\bR^d\times\bR^d$,
$$
\rho_f(t,x):=\int_{\bR^d}f(t,x,\xi)d\xi\,,
$$
and $\{\cdot,\cdot\}$ is the Poisson bracket such that
$$
\{x_k,x_l\}=\{\xi_k,\xi_l\}=0\,,\quad\{\xi_k,x_l\}=\de_{kl}\,.
$$

Next we define a way to measure the convergence rate of $R_\hbar$ to $f$ in the limit as $\hbar\to 0$. It involves a quantity which is intermediate between the notion of Monge-Kantorovich distance of exponent $2$ between Borel probability 
measures on $\bR^d\times\bR^d$ and the pseudo-distance $MK_2^\eps$ between density operators defined in \cite{FGMouPaul} (Definition 2.2). First we define a notion of coupling between distribution functions in statistical mechanics and
density operators in quantum mechanics.

\begin{Def}
Let $p\equiv p(x,\xi)$ be a probability density on $\bR^d\times\bR^d$, and let $R\in\cD(\fH)$. A coupling of $p$ and $R$ is a measurable function $Q:\,(x,\xi)\mapsto Q(x,\xi)$ defined a.e. on $\bR^d\times\bR^d$ and with values in $\cL(\fH)$ 
s.t. $Q(x,\xi)=Q(x,\xi)^*\ge 0$  for a.e. $(x,\xi)\in\bR^d\times\bR^d$, and
$$
\left\{
\ba
{}&\Tr(Q(x,\xi))=p(x,\xi)\hbox{ for a.e. }(x,\xi)\in\bR^d\times\bR^d\,,
\\
&\int_{\bR^d\times\bR^d}Q(x,\xi)dxd\xi=R\,.
\ea
\right.
$$
The set of all such functions is denoted by $\cC(p,R)$.
\end{Def}

\smallskip
Notice that $\cC(p,R)$ is nonempty, since the function $(x,\xi)\mapsto p(x,\xi)R$ belongs to $\cC(p,R)$.

\bigskip
Mimicking the definition of Monge-Kantorovich distances, we next define the pseudo-distance between $R_\hbar$ and $f$ in terms of an appropriate ``cost function'' analogous to the quadratic cost function used in optimal transport.

\begin{Def}\label{ehbar}
For each probability density $p\equiv p(x,\xi)$ on $\bR^d\times\bR^d$ and each $R\in\cD(\fH)$, we set
$$
E_\hbar(p,R):=\left(\inf_{Q\in\cC(p,R)}\int_{\bR^d\times\bR^d}\Tr(c_\hbar(x,\xi)Q(x,\xi))dxd\xi\right)^{1/2}\in[0,+\infty]\,,
$$
where the transportation cost $c_\hbar$ is the function of $(x,\xi)$ with values in the set of unbounded operators on $\fH=L^2(\bR^d_y)$ defined by the formula
$$
c_\hbar(x,\xi):=\tfrac12(|x-y|^2+|\xi+i\hbar\grad_y|^2)\,.
$$
\end{Def}

Before going further, we briefly discuss some basic properties of $E_\hbar$. First we recall a few elementary facts concerning T\"oplitz quantization on $\bR^d$. For each $z=x+i\xi\in\bC^{d}$, we denote 
$$
|z,\hbar\ra:\,y\mapsto(\pi\hbar)^{-d/4}e^{-|y-x|^2/2\hbar}e^{i\xi\cdot(y-x)/\hbar}\,,
$$
and we designate by $|z,\hbar\ra\la z,\hbar|$ the orthogonal projection on the line $\bC|z,\hbar\ra$ in $\fH$. An elementary computation shows that
$$
\| |z,\hbar\ra\|_\fH=1\,.
$$
For each Borel probability measure on $\bR^d\times\bR^d$, we set
$$
\Op^T_\hbar(\mu):=\frac1{(2\pi\hbar)^d}\int_{\bR^d\times\bR^d}|x+i\xi,\hbar\ra\la x+i\xi,\hbar|\mu(dxd\xi)
$$
and we recall that
\be\label{decomp}
\Op^T_\hbar(1)=I_\fH\,.
\ee
If $\phi$ is a polynomial of degree $\le 2$, one has
\be\lb{OpTQuad}
\Op^T_\hbar(\phi(x))=\phi(x)+\tfrac14\hbar(\Dlt\phi)I_\fH\,,\quad\Op^T_\hbar(\phi(\xi))=\phi(-i\hbar\grad_x)+\tfrac14\hbar(\Dlt\phi)I_\fH\,,
\ee
according to formula (48) in \cite{FGMouPaul}.

We also recall the definition of the Wigner and Husimi transforms of a density operator on $\fH$. If $R\in\cD(\fH)$ with integral kernel $r$, its Wigner transform at scale $\hbar$ is the function on $\bR^d\times\bR^d$ defined by the formula
$$
W_\hbar[R](x,\xi):=\frac1{(2\pi)^d}\int_{\bR^d}e^{-i\xi\cdot y}r(x+\tfrac12\hbar y,x-\tfrac12\hbar y)dy\,.
$$
The Husimi transform of $R$ is
$$
\widetilde W_\hbar[R]:=e^{\hbar\Dlt_{x,\xi}/4}W_\hbar[R]
$$
and we recall that
$$
\widetilde W_\hbar[R]\ge 0\,,
$$
while
$$
\int_{\bR^d\times\bR^d}\widetilde W_\hbar[R](x,\xi)dxd\xi=\int_{\bR^d\times\bR^d}W_\hbar[R](x,\xi)dxd\xi=\Tr(R)=1\,.
$$
In particular, $\widetilde W_\hbar[R]$ is a probability density on $\bR^d\times\bR^d$ for each $R\in\cD(\fH)$.

\begin{Rmk}[$\widetilde{W_\hbar}{[R]}$ uniquely determines $R$]\label{husforever}
Let $R$ be a density operator on $\fH=L^2(\bR^d)$ with integral kernel $r\equiv r(y,y')$. Elementary computations show that, for each $x,\xi\in\bR^d$, one has
$$
\widetilde W_\hbar[R](x,\xi)=\frac{(\pi\hbar)^{d/2}e^{-|x|^2/\hbar}}{2^d(2\pi\hbar)^{2d}}J(x,\xi)\,,
$$
where
$$
J(x,\xi):=\iint r(y,y')e^{-(|y|^2+|y'|^2)/2\hbar}e^{(x\cdot(y+y')-i\xi\cdot(y-y'))/\hbar}dydy'\,.
$$
Since $R$ is a Hilbert-Schmidt operator on $\fH$, its kernel $r\in L^2(\bR^d\times\bR^d)$ and therefore $J$ extends as an entire holomorphic function of $(x,\xi)\in\bC^d\times\bC^d$. Therefore $J$ is uniquely determined by its restriction
to $\bR^d\times\bR^d$, i.e. by $\widetilde W_\hbar[R]$. Denoting by $\cF$ the Fourier transformation, we observe that
$$
J(-ix,\xi)=\cF[r\exp(-(|y|^2+|y'|^2)/2\hbar)]((x+\xi)/\hbar,(x-\xi)/\hbar)\,,
$$
and conclude that $J$ uniquely determines in turn $r$ by Fourier inversion. Therefore, the operator $R$ is uniqueley determined by its Husimi transform. See also Lemma A.2.1 in \cite{BSimon} for a more general result of the same type.
\end{Rmk}

If $\mu$ is a Borel probability measure on $\bR^d\times\bR^d$, then
\be\lb{TrOpTR}
\Tr(\Op^T_\hbar(\mu)R)=\int_{\bR^d\times\bR^d}\widetilde W_\hbar[R](x,\xi)\mu(dxd\xi)\,,
\ee
according to formula (54) in \cite{FGMouPaul}. In particular, for each polynomial of degree $\le 2$, one has
\be\lb{TrOpTQuad}
\Tr((\phi(x)+\phi(-i\hbar\grad_x))\Op^T_\hbar(\mu))=\int_{\bR^d\times\bR^d}(\phi(x)+\phi(\xi))\mu(dxd\xi)+\tfrac12\hbar\Dlt\phi\,,
\ee
see formula (55) in \cite{FGMouPaul}. (See also Lemma 5.4 in \cite{Lerner}, and more generally \cite{BeSh,Lerner,LernerBook} for a more complete discussion on quantization and symbolic calculus.)

In addition, we shall need the following properties of $E_\hbar$.

\begin{Thm}\lb{T-PtyE}
For each probability density $p$ on $\bR^d\times\bR^d$ such that
$$
\int_{\bR^d\times\bR^d}(|x|^2+|\xi|^2)p(x,\xi)dxd\xi<\infty
$$
and each $R\in\cD(\fH)$, one has
$$
E_\hbar(p,R)^2\ge\tfrac12d\hbar\,.
$$
(1) Let $R_\hbar=\Op^T_\hbar((2\pi\hbar)^d\mu)$, where $\mu$ is a Borel probability measure on $\bR^d\times\bR^d$. Then
$$
E_\hbar(p,R_\hbar)^2\le\MKd(p,\mu)^2+\tfrac12d\hbar\,.
$$
(2) For each $R\in\cD(\fH)$, one has
$$
E_\hbar(p,R)^2\ge\MKd(p,\widetilde W_\hbar[R])^2-\tfrac12d\hbar\,.
$$
(3) If $R_\hbar\in\cD(\fH)$ and $W_\hbar[R_\hbar]\to\mu$ in $\cS'(\bR^d\times\bR^d)$ as $\hbar\to 0$, then $\mu$ is a Borel probability measure on $\bR^d\times\bR^d$ and one has
$$
\MKd(p,\mu)\le\varliminf_{\hbar\to 0}E_\hbar(p,R_\hbar)\,.
$$
\end{Thm}

The main result in this section is a quantitative estimate of the convergence rate of the solution $R_\hbar$ of the Hartree equation to the solution $f$ of the Vlasov equation in terms of the pseudo-distance $E_\hbar$ and, as a by-product, 
in terms of the Monge-Kantorovich distance between $f$ and the Husimi function of $R_\hbar$. We recall that this limit has been formulated in terms of the Wigner transform of $R_\hbar$ in \cite{LionsPaul} (Theorem IV.2). This is the
right vertical arrow in the diagram of section \ref{S-Intro}.

\begin{Thm}\lb{T-HV}
Let $V$ be an even, real-valued, bounded function of class $C^{1,1}$ on $\bR^d$. Denote by $L$ the Lipschitz constant of $\grad V$ and let $\Lambda=1+\max(1,4L^2)$. Let $f^{in}$ be a probability density on $\bR^d\times\bR^d$ such 
that
$$
\int_{\bR^d\times\bR^d}(|x|^2+|\xi|^2)f^{in}(x,\xi)dxd\xi<\infty\,,
$$
and let $f$ be the solution of the Vlasov equation (\ref{Vlasov}) with initial data $f^{in}$. Let $R^{in}_\hbar\in\cD(\fH)$, and let $R_\hbar$ be the solution of (\ref{Hartree}) with initial data $R^{in}_\hbar$. 

Then, for each $t\ge 0$, one has
$$
E_\hbar(f(t),R_\hbar(t))^2\le e^{\Lambda t}E_\hbar(f^{in},R^{in}_\hbar)^2\,.
$$
In particular, for all $t\ge 0$, one has
$$
\MKd(f(t),\widetilde W_\hbar[R_\hbar(t)])^2\le e^{\Lambda t}E_\hbar(f^{in},R_\hbar^{in})^2+\tfrac12d\hbar\,.
$$

Moreover, if $R^{in}_\hbar=\Op^T_\hbar((2\pi\hbar)^d\mu^{in})$, where $\mu^{in}$ is a Borel probability measure on $\bR^d\times\bR^d$, then
$$
\MKd(f(t),\widetilde W_\hbar[R_\hbar(t)])^2\le e^{\Lambda t}\left(\MKd(f^{in},\mu^{in})^2+\tfrac12d\hbar\right)+\tfrac12d\hbar\,.
$$
\end{Thm}

\smallskip
We do not assume that the initial data $f^{in}$ is smooth, or that $[\Dlt,R^{in}_\hbar]$ is a trace-class operator on $\fH$. Hence the solutions $t\mapsto f(t,x,\xi)$ and $t\mapsto R_\hbar(t)$ of the Vlasov and the Hartree equations considered
in the statement above are not classical solutions, but weak solutions. Since $V$ is of class $C^1$ with Lipschitz continuous gradient on $\bR^d$, the characteristic flow of the Vlasov equation (\ref{Vlasov}) is defined globally by the 
Cauchy-Lipschitz theorem, and the Cauchy problem (\ref{Vlasov}) has a unique solution obtained as the push-forward of $f^{in}$ under this flow. The solution $t\mapsto R_\hbar(t)$ of the Cauchy problem (\ref{Hartree}) is obtained similarly
as the conjugate of $R^{in}_\hbar$ by some time-dependent unitary operator (see \cite{BoveDPF}, especially Proposition 4.3 there). This remark also applies to the statements of Theorems \ref{T-NSV} and \ref{T-SL} below and will 
not be repeated. 

\smallskip
Notice that, because of statement (3) in Theorem \ref{T-PtyE}, if the Wigner transform of $R_\hbar(t)$ satisfies 
$$
W_\hbar[R_\hbar(t)]\to\mu(t)\quad\hbox{ in }\cS'(\bR^d\times\bR^d)
$$
for all $t\ge 0$ as $\hbar\to 0$, then the last inequality in Theorem \ref{T-HV} leads to
$$
\MKd(f(t),\mu(t))\le\MKd(f^{in},\mu^{in})e^{\L t/2}\,,\qquad t\ge 0\,.
$$
Since $\mu$ is a weak solution of the Vlasov equation by Theorem IV.2 in \cite{LionsPaul}, this inequality is precisely the analogue of Dobrushin's inequality in \cite{Dobru}, formulated in termes of the Monge-Kantorovich distance of 
exponent $2$, instead of the Monge-Kantorovich distance of exponent $1$ as in \cite{Dobru}. (That the Monge-Kantorovich distance of exponent $2$ can be used in Dobrushin's argument has been observed by Loeper \cite{Loeper}; 
see also \cite{FGMouPaul} for an extension to Monge-Kantorovich distances of arbitrary exponents.)


\subsection{From the $N$-Body Schr\"odinger Equation to the Vlasov Equation}\label{FNSTV}


In this section, we explain how the Vlasov equation, i.e. the mean-field theory of large particle systems in classical mechanics, can be deduced from the $N$-body Schr\"odinger equation in the limit of large $N$ and small $\hbar$. In other
words, we combine the classical limit discussed in the previous section with the mean-field limit in quantum mechanics, in which the Hartree equation is deduced from the linear $N$-body Schr\"odinger equation. Since the mean-field limit
of the $N$-body Schr\"odinger equation is uniform in the limit as $\hbar\to 0$ according to Theorem 2.4 in \cite{FGMouPaul}, this combined limit is a straightforward consequence of the quantitative estimate for the classical limit obtained
in Theorem \ref{T-HV} above, at least when formulated in terms of Monge-Kantorovich distances and Husimi transforms.

In this section, we propose a slightly different approach, and estimate directly the difference between the single-particle marginal of the $N$-particle density operator and the Vlasov solution in terms of the $E_\hbar$ functional. This estimate
combines the ideas used in the proof of Theorem \ref{T-HV} above, with those of \cite{FGMouPaul}.

Before stating our main result, we introduce some elements of notation pertaining to large particle systems.

Firstly, the particles considered in the present work are indistinguishable. The mathematical formulation of this property is the following symmetry condition. For each $N>1$, let $\fS_N$ be the group of permutations of the set $\{1,\ldots,N\}$.
For each $\si\in\fS_N$ and each $X_N:=(x_1,\ldots,x_N)\in(\bR^d)^N$, we denote
$$
\si\cdot X_N:=(x_{\si(1)},\ldots,x_{\si(N)})\,.
$$
With $\fH:=L^2(\bR^d)$ as in the previous section, we set $\fH_N:=\fH^{\otimes N}=L^2((\bR^d)^N)$. For each $\psi\in\fH_N$ and each $\si\in \fS_N$, we set
$$
U_\si\psi_N(X_N):=\psi_N(\si\cdot X_N)\,.
$$
Obviously
$$
U_\si^*=U_{\si^{-1}}=U_\si^{-1}
$$
so that $U_\si$ is a unitary operator on $\fH_N$. A density operator for a system of $N$ indistinguishable particles is an element $R_N\in\cD(\fH_N)$ satisfying the symmetry  relation
$$
U_\si R_NU_\si^*=R\quad\hbox{ for each }\si\in\fS_N\,.
$$
The subset of elements of $\cD(\fH_N)$ satisfying this symmetry condition is henceforth denoted $\cD^s(\fH_N)$.

Secondly, we recall the notion of $n$-body marginal of an element of $\cD(\fH_N)$ for all $n=1,\ldots,N$. For each $R_N\in\cD(\fH_N)$ and each $n=1,\ldots,N$, we define $R_N^\mathbf{n}$ to be the unique element in $\cD(\fH_n)$
such that
\be\lb{nQMargi}
\Tr_{\fH_n}(R_N^\mathbf{n}(A_1\otimes\ldots\otimes A_n))=\Tr_{\fH_N}(R_N(A_1\otimes\ldots\otimes A_n\otimes I_\fH\otimes\ldots\otimes I_\fH))
\ee
for each $A_1,\ldots,A_n\in\cL(\fH)$. Equivalently
$$
\Tr_{\fH_n}(R_N^\mathbf{n}B_n)=\Tr_{\fH_N}(R_N(B_n\otimes I_{\fH_{N-n}}))
$$
for each $B_n\in\cL(\fH_n)$, using the identification $\fH_N\simeq\fH_n\otimes\fH_{N-n}$. Clearly
$$
R_N\in\cD^s(\fH_N)\Rightarrow R_N^\mathbf{n}\in\cD^s(\fH_n)\,.
$$

At this point, we introduce the $N$-particle dynamics. Let the interaction potential $V$ be an even, real-valued function of class $C^{1,1}$ on $\bR^d$, and consider the $N$-particle quantum Hamiltonian
\be\lb{QHamN}
\cH_{\hbar,N}:=\sum_{j=1}^N-\tfrac12\hbar^2\Dlt_{x_j}+\frac1{2N}\sum_{j,k=1}^NV(x_j-x_k)\,.
\ee
Assuming that the state of the $N$-particle system at time $t=0$ is given by the symmetric density operator $R^{in}_{\hbar,N}\in\cD^s(\fH_N)$, the state at time $t>0$ of that same system is given by the propagated density operator
\be\lb{NqDens(t)}
R_{\hbar,N}(t)=e^{-it\cH_{\hbar,N}/\hbar}R^{in}_{\hbar,N}e^{it\cH_{\hbar,N}/\hbar}\,,\qquad t\ge 0\,.
\ee
Observe that $U_\si\cH_{\hbar,N}=\cH_{\hbar,N}U_\si$ for each $\si\in\fS_N$, so that 
$$
U_\si e^{it\cH_{\hbar,N}/\hbar}=e^{it\cH_{\hbar,N}/\hbar}U_\si
$$ 
for each $t\in\bR$ and each $\si\in\fS_N$. As a result, the symmetry property of $R^{in}_{\hbar,N}$ is propagated by the dynamics, i.e.
$$
U_\si R_{\hbar,N}(t)U_\si^*=R_{\hbar,N}(t)\quad\hbox{ for each }t\ge 0\hbox{ and each }\si\in\fS_N\,.
$$

\begin{Thm}\lb{T-NSV}
Let $V$ be an even, real-valued function of class $C^{1,1}$ on $\bR^d$. Denote by $L$ the Lipschitz constant of $\grad V$ and let $\Gamma=2+\max(4L^2,1)$. Let $R^{in}_{\hbar,N}\in\cD^s(\fH_N)$ and let $R_{\hbar,N}(t)$ be given by 
(\ref{NqDens(t)}) for each $t\ge 0$. Let $f^{in}$ be a probability density on $\bR^d\times\bR^d$ such that
$$
\int_{\bR^d\times\bR^d}(|x|^2+|\xi|^2)f^{in}(x,\xi)dxd\xi<\infty\,,
$$
and let $f$ be the solution of the Vlasov equation (\ref{Vlasov}) with initial data $f^{in}$. Then
$$
\frac1nE_\hbar(f(t)^{\otimes n},R_{\hbar, N}^\mathbf{n}(t))^2\le\frac1NE_\hbar((f^{in})^{\otimes N},R^{in}_{\hbar, N})^2e^{\Gamma t}+\frac{4\|\grad V\|^2_{L^\infty}}{N-1}\frac{e^{\Gamma t}-1}{\Gamma}
$$
for each $n=1,\ldots,N$ and each $t\ge 0$. In particular, for each $n=1,\ldots,N$ and each $t\ge 0$, one has
$$
\ba
\frac1n\MKd(f(t)^{\otimes n},\widetilde W_\hbar[R_{\hbar, N}^\mathbf{n}(t)])^2
\\
\le\frac1NE_\hbar((f^{in})^{\otimes N},R^{in}_{\hbar, N})^2e^{\Gamma t}+\frac{4\|\grad V\|^2_{L^\infty}}{N-1}\frac{e^{\Gamma t}-1}{\Gamma}+\tfrac12d\hbar&\,.
\ea
$$
If moreover $R_{\hbar,N}^{in}$ is the T\"oplitz operator at scale $\hbar$ with symbol $(2\pi\hbar)^{dN}\mu^{in}_N$ where $\mu^{in}_N$ is a symmetric Borel probability measure on $(\bR^d)^N$, i.e. 
$$
R_{\hbar,N}^{in}=\Op^T_\hbar((2\pi\hbar)^{dN}\mu^{in}_N)\,,
$$
one has
$$
\ba
\frac1n\MKd(f(t)^{\otimes n},\widetilde W_\hbar[R_{\hbar, N}^\mathbf{n}(t)])^2
\\
\le\left(\frac1N\MKd((f^{in})^{\otimes N},\mu^{in}_N)^2+\tfrac12d\hbar\right)e^{\Gamma t}+\frac{4\|\grad V\|^2_{L^\infty}}{N-1}\frac{e^{\Gamma t}-1}{\Gamma}+\tfrac12d\hbar&\,.
\ea
$$
\end{Thm}

The result in Theorem \ref{T-NSV} is the diagonal arrow in the diagram of section \ref{S-Intro}. It is natural to consider the quantity $\frac1NE_\hbar(F^{in}_N,R^{in}_{\hbar,N})^2$ since the square of the ``distance" $E_\hbar$ between
symmetric $N$-particle densities (quantum and classical) grows linearly with the particle number $N$.


\subsection{From the $N$-Body Schr\"odinger Equation to the $N$-Body Liouville Equation}


As a by-product of the methods introduced to prove Theorems \ref{T-HV} and \ref{T-NSV}, we establish the validity of the classical limit of the $N$-body Schr\"odinger equation to the $N$-body Liouville equation as $\hbar\to 0$, i.e. the 
left vertical arrow in the diagram of section \ref{S-Intro}. This limit is of course well known for each $N\ge 1$ (see for instance Theorem IV.2 in \cite{LionsPaul}); the novelty in the approach presented here is that this limit is proved to 
be \textit{uniform as} $N\to\infty$. More precisely, we seek to compare the evolved density operator $R_{\hbar,N}(t)$ defined by (\ref{NqDens(t)}) in the previous section \ref{FNSTV}, and the solution $F_N(t,X_N,\Xi_N)$ of the Liouville equation
\be\lb{CPLiou}
\left\{
\ba
{}&\d_t F_N+\sum_{k=1}^N\xi_k\cdot\grad_{x_k}F_N-\frac1{N}\sum_{k,l=1}^N\grad V(x_k-x_l)\cdot\nabla_{\xi_k}F_N=0\,,
\\
&F_N\rstr_{t=0}=F_N^{in}\,.
\ea
\right.
\ee
In other words
\be\lb{CPLiouHam}
\left\{
\ba
{}&\d_t F_N+\{\bH_N,F_N\}_N=0\,,
\\
&F_N\rstr_{t=0}=F_N^{in}\,,
\ea
\right.
\ee
where 
\be\lb{CHamN}
\bH_N(X_N,\Xi_N):=\sum_{j=1}^N\tfrac12|\xi_j|^2+\tfrac1{2N}\sum_{j,k=1}^NV(x_j-x_k)
\ee
and where $\{\cdot,\cdot\}_N$ is the $N$-particle Poisson bracket on $(\bR^d\times\bR^d)^N$ defined by
$$
\{x_j,x_k\}_N=\{\xi_j,\xi_k\}_N=0\,,\quad\{\xi_j,x_k\}_N=\de_{jk}\,,\qquad j,k=1,\ldots,N\,.
$$

We obtain a uniform (in $N$) bound for the quantity
$$
\frac1NE_\hbar(F_N(t),R_{\hbar,N}(t))^2
$$
in terms of its value for $t=0$. For $n=1,\dots,N$, let $F_N^\mathbf{n}$ be the $n$-th marginal of $F_N$, as in the statement of Theorem \ref{T-NSV}.

\begin{Thm}\label{T-SL}
Let $V$ be an even, real-valued function of class $C^{1,1}$ on $\bR^d$ with bounded gradient. Denote by $L$ the Lipschitz constant of $\grad V$ and let $\L=1+\max(4L^2,1)$. Let $R^{in}_{\hbar,N}\in\cD^s(\fH_N)$ and let $R_{\hbar,N}(t)$ 
be given by (\ref{NqDens(t)}) for each $t\ge 0$. Let $F^{in}_N$ be a symmetric probability density on $(\bR^d\times\bR^d)^N$ such that
$$
\int_{(\bR^{d}\times\bR^{d})^N}(|X_N|^2+|\Xi_N|^2)F^{in}_N(X_N,\Xi_N)dX_Nd\Xi_N<\infty\,,
$$
and let $F_N$ be the solution of the $N$-body Liouville equation (\ref{CPLiou}) with initial data $F^{in}_N$. 

Then, for all $t\ge 0$ and all $n=1,\ldots,N$, 
$$
\frac1nE_\hbar(F^\mathbf{n}_N(t),R^\mathbf{n}_{\hbar,N}(t))^2\le\frac1NE_\hbar(F^{in}_N,R^{in}_{\hbar,N})^2e^{\L t}\,.
$$
In particular, 
$$
\frac1n\MKd(F^\mathbf{n}_N(t),\widetilde W_\hbar[R^\mathbf{n}_{\hbar,N}(t)])^2\le\frac1NE_\hbar(F^\mathbf{n}_N(t),R^\mathbf{n}_{\hbar,N}(t))^2e^{\L t}+\tfrac12d\hbar.
$$

Moreover if $R^{in}_{\hbar,N}=\Op^T_\hbar((2\pi\hbar)^{dN}\mu^{in}_N)$ where $\mu^{in}_N$ is a symmetric Borel probability measure on $(\bR^d\times\bR^d)^N$, then
$$
\ba
\frac1n\MKd(F^\mathbf{n}_N(t),\widetilde W_\hbar[R^\mathbf{n}_{\hbar,N}(t)])^2&
\\
\le\left(\frac1N\MKd(F^{in}_N,\mu^{in}_N)^2+\tfrac12d\hbar\right)e^{\L t}+\tfrac12d\hbar&\,,
\ea
$$
for all $t\ge 0$ and each $n=1,\ldots,N$. In the particular case where $\mu^{in}_N=F^{in}_N$, one finds that
$$
\frac1n\MKd(F^\mathbf{n}_N(t),\widetilde W_\hbar[R^\mathbf{n}_{\hbar,N}(t)])^2\le\tfrac12d\hbar(1+e^{\L t})
$$
for all $t\ge 0$ and each $n=1,\ldots,N$.
\end{Thm}

\smallskip
Together with Theorem 2.4 in \cite{FGMouPaul}, Theorem \ref{T-SL} shows that both the upper horizontal and left vertical arrows in the diagram of section \ref{S-Intro} correspond to uniform limits. This uniformity explains why the mean-field,
large $N$ limit and the classical, small $\hbar$ limit can be taken simultaneously in the quantum $N$-body problem, in the case of interaction potentials with Lipschitz continuous gradient.

\subsection{Remarks}


Before starting with the proofs of Theorems \ref{T-HV}, \ref{T-NSV} and \ref{T-SL}, it is perhaps interesting to mention that the arguments used in these proofs can be adapted to the case where the Hamiltonian under consideration includes
an external (i.e. noninteracting) potential acting on each individual particle. For instance, our analysis in the proof of Theorem \ref{T-NSV} can be adapted to the case where the quantum $N$-body Hamiltonian is 
$$
\cH_{\hbar,N}:=\sum_{j=1}^N(-\tfrac12\hbar^2\Dlt_{x_j}+W(x_j))+\frac1{2N}\sum_{j,k=1}^NV(x_j-x_k)\,,
$$
assuming that $W$ is a real-valued function with bounded and Lipschitz continuous gradient such that $-\tfrac12\hbar^2\Dlt+W$ defines a self-adjoint operator on $L^2(\bR^d)$. The resulting Vlasov equation would be, in that case
$$
\d_tf+\{\tfrac12|\xi|^2+W(x)+V\star_x\rho_f(t,x),f\}=0
$$
instead of (\ref{Vlasov}). The convergence rate in Theorem \ref{T-NSV} would include the Lipschitz constant of the gradient of the external potential $W$.

\smallskip
For applications to physically realistic situations, it would be important to extend the results above to the case of singular interaction potentials. The Coulomb potential is of course the most relevant singular interaction for applications to atomic 
physics or quantum chemistry. One approach to this question would be to start with a mollified potential, and to remove the mollification as the particle number $N\to\infty$. This is the approach used for instance in \cite{HaurayJabin,PicklLaza}
--- see also the references therein. Mollifying the Coulomb potential can be understood in some sense as replacing point particles with spherical particles with some positive radius that vanishes in the large $N$ limit. The arguments in the 
present paper should be rather easily adapted to truncations of this type assuming that the particle radius vanishes slowly enough as $N\to\infty$ --- for instance, if the particle radius is of order $\gg C/\sqrt{\ln N}$ for some $C>0$. With this type 
of truncation one can hope to derive the Vlasov-Poisson equation following the line of Theorem \ref{T-NSV}. This truncation amounts to considering point particle systems so rarefied that the probability of observing such configuration in the 
spatial domain vanishes as $N\to\infty$. For the mean-field limit in classical mechanics, more satisfying results have been obtained recently, with much more realistic dependence of the particle radius in terms of $N$ (see 
\cite{HaurayJabin,PicklLaza,Laza}). Whether these results can be generalized to the quantum setting considered here remains an open question and would most likely require some additional new ideas.

\smallskip
We conclude this section with a few words on the method used in the proof of Theorems \ref{T-HV}, \ref{T-NSV} and \ref{T-SL}. The key idea is\footnote{In the words of one of the referees.} ``to double the system size by considering two kinds 
of particles, half of them being classical, the other half being quantum''. Then one writes a ``transport equation'' for the joint dynamics of these two kinds of particles, whose projections in the classical (resp. the quantum) part of the space is 
the governing equation for the classical (resp. the quantum) system of particles: see equations (\ref{PbCCoupl}) and (\ref{EqNCoupl}) below. The estimates on the pseudo-distance $E_\hbar$ between the quantum and the classical densities 
in Theorems \ref{T-HV}, \ref{T-NSV} and \ref{T-SL} are then obtained by considering appropriate ``moments'' of the solution of the transport equations (\ref{PbCCoupl}) and (\ref{EqNCoupl}) involving the cost function $c_\hbar$.


\section{Proof of Theorem \ref{T-PtyE}}


\subsection{Proof of the general lower bound}


Set $A_j=(x_j-y_j)$ and $B_j=\xi_j+i\hbar\d_{y_j}$; both $A_j$ and $B_j$ are self-adjoint unbounded operators on $\fH$, and one has
$$
A_j^2+B_j^2=(A_j+iB_j)^*(A_j+iB_j)+i[B_j,A_j]\ge i[B_j,A_j]=i(-i\hbar)=\hbar
$$
for each $j=1,\ldots,d$. Thus
$$
c_\hbar(x,\xi)=\tfrac12\sum_{j=1}^d(A_j^2+B_j^2)\ge\tfrac12d\hbar\,.
$$
In particular, for each $Q\in\cC(p,R)$, one has
$$
\int_{\bR^d\times\bR^d}\Tr(Q(x,\xi)c_\hbar(x,\xi))dxd\xi\ge\tfrac12d\hbar\int_{\bR^d\times\bR^d}\Tr(Q(x,\xi))dxd\xi=\tfrac12d\hbar\,.
$$
Minimizing the left hand side of this inequality over $\cC(p,R)$ leads to the announced inequality.

\subsection{Proof of (1)}


We shall use the following intermediate result.

\begin{Lem}\lb{L-CouplT}
Let $p$ be a probability density and $\mu$ a Borel probability measure on $\bR^d\times\bR^d$, and let $q$ be a coupling for $p$ and $\mu$. Then $q$ is a measurable map $(x,\xi)\mapsto q(x,\xi)$ defined a.e. on $\bR^d\times\bR^d$ with 
values in the set of positive Borel measures on $\bR^d\times\bR^d$, and the map 
$$
Q_\hbar:\,(x,\xi)\mapsto\Op^T_\hbar((2\pi\hbar)^dq(x,\xi))
$$
belongs to $\cC(p,\Op^T_\hbar((2\pi\hbar)^d\mu))$.
\end{Lem}

\smallskip
In other words, for each coupling $q$ of $p$ and $\mu$, the T\"oplitz operator 
$$
Q_\hbar(x,\xi)=\Op^T_\hbar((2\pi\hbar)^dq(x,\xi))\hbox{ belongs to }\cC(p,\Op^T_\hbar((2\pi\hbar)^d\mu))\,.
$$
Therefore
$$
E_\hbar(p,\Op^T_\hbar((2\pi\hbar)^d\mu))^2\le\int_{\bR^d\times\bR^d}\Tr(c_\hbar(x,\xi)\Op^T_\hbar((2\pi\hbar)^dq(x,\xi)))dxd\xi\,.
$$
For each $z\in\bR^d$, consider the quadratic polynomial 
$$
y\mapsto\g(z,y):=\tfrac12|z-y|^2\,.
$$
Then 
$$
c_\hbar(x,\xi)=\g(x,y)+\g(\xi,-i\hbar\grad_y)\,,
$$
and applying formula (\ref{TrOpTQuad}) shows that
$$
\ba
\Tr(c_\hbar(x,\xi)\Op^T_\hbar((2\pi\hbar)^dq(x,\xi)))\!=\!&\int_{\bR^d\times\bR^d}(\g(x,y)\!+\!\g(\xi,\eta))\mu(dyd\eta)\!+\!\tfrac12\hbar\Dlt\g(z,\cdot)
\\
=&\int_{\bR^d\times\bR^d}(\g(x,y)+\g(\xi,\eta))\mu(dyd\eta)+\tfrac12d\hbar\,.
\ea
$$
Hence
$$
E_\hbar(p,\Op^T_\hbar((2\pi\hbar)^d\mu))^2\le\int_{\bR^d\times\bR^d}(\g(x,y)+\g(\xi,\eta))\mu(dyd\eta)+\tfrac12d\hbar
$$
for each $q\in\Pi(p,\mu)$, and minimizing the right hand side of this inequality as $q$ runs through $\Pi(p,\mu)$ leads to the announced inequality.

\smallskip
It remains to prove Lemma \ref{L-CouplT}. This is a variant of Lemma 4.1 in \cite{FGMouPaul}.

\begin{proof}[Proof of Lemma \ref{L-CouplT}]
First, we observe that, for each $\phi\in C_b(\bR^d\times\bR^d)$
$$
\left|\iint_{\Om\times(\bR^d\times\bR^d)}\phi(y,\eta)q(dxd\xi dyd\eta)\right|\le\|\phi\|_{L^\infty}\iint_\Om p(x,\xi)dxd\xi=0
$$
if $\Om$ is Lebesgue-negligible in $\bR^d\times\bR^d$. Hence, for each $\phi\in C_b(\bR^d\times\bR^d)$, the bounded Borel measure on $\bR^d\times\bR^d$ defined by
$$
\Om\mapsto\iint_{\Om\times(\bR^d\times\bR^d)}\phi(y,\eta)q(dxd\xi dyd\eta)
$$
is absolutely continuous with respect to the Lebesgue measure on $\bR^d\times\bR^d$. By the Radon-Nikodym theorem, this Borel measure has a density with respect to the Lebesgue measure on $\bR^d\times\bR^d$ which is precisely
$$
(x,\xi)\mapsto\int_{\bR^d\times\bR^d}\phi(y,\eta)q(x,\xi,dyd\eta)
$$
where $q(x,\xi,\cdot)$ is the sought map. 

Since $q(x,\xi)$ is a probability measure for a.e. $(x,\xi)\in\bR^d\times\bR^d$, one has
$$
Q_\hbar(x,\xi)=\Op^T_\hbar((2\pi\hbar)^dq(x,\xi))=\Op^T_\hbar((2\pi\hbar)^dq(x,\xi))^*=Q_\hbar(x,\xi)^*\ge 0\,,
$$
and
$$
\Tr(Q_\hbar(x,\xi))=\Tr(\Op^T_\hbar((2\pi\hbar)^dq(x,\xi)))=\int_{\bR^d\times\bR^d}q(x,\xi,dyd\eta)=p(x,\xi)\,.
$$
On the other hand
$$
\ba
\int_{\bR^d\times\bR^d}Q_\hbar(x,\xi)dxd\xi=&\int_{\bR^d\times\bR^d}\Op^T_\hbar((2\pi\hbar)^dq(x,\xi))dxd\xi
\\
=&\Op^T_\hbar\left((2\pi\hbar)^d\int_{\bR^d\times\bR^d}q(x,\xi)dxd\xi\right)=\Op^T_\hbar((2\pi\hbar)^d\mu)\,.
\ea
$$
Hence $Q_\hbar\in\cC(p,\Op^T_\hbar((2\pi\hbar)^d\mu))$.
\end{proof}

\subsection{Proof of (2)}


Let $a,b\in C_b(\bR^d\times\bR^d)$ satisfy
\be\lb{Ineq-ab}
a(x,\xi)+b(y,\eta)\le\tfrac12(|x-y|^2+|\xi-\eta|^2)=:\Ga(x,\xi,y,\eta)
\ee
for each $x,y,\xi,\eta\in\bR^d$. Then
$$
a(x,\xi)\Op^T_\hbar(1)+\Op^T_\hbar(b)\le\Op^T_\hbar(\Ga(x,\xi,\cdot,\cdot))=c_\hbar(x,\xi)+\tfrac12d\hbar I_\fH\,,
$$
where the last equality follows from formulas (\ref{OpTQuad}). 

Then, for each $Q\in\cC(f,R)$
$$
\ba
\int_{\bR^d\times\bR^d}\Tr(c_\hbar(x,\xi)Q(x,\xi))dxd\xi&
\\
\ge\int_{\bR^d\times\bR^d}a(x,\xi)f(x,\xi)dxd\xi+\Tr(\Op^T_\hbar(b)R)-\tfrac12d\hbar&
\\
=\int_{\bR^d\times\bR^d}a(x,\xi)f(x,\xi)dxd\xi+\int_{\bR^d\times\bR^d}b(y,\eta)\widetilde W_\hbar[R](y,\eta)dyd\eta-\tfrac12d\hbar&\,,
\ea
$$
where the last equality follows from formula (\ref{TrOpTR}).

Minimizing the left-hand side of this inequality over $Q\in\cC(f,R)$, and maximizing over all $a,b\in C_b(\bR^d\times\bR^d)$ satisfying (\ref{Ineq-ab}), one has
$$
\ba
E_\hbar(f,R)^2+\tfrac12d\hbar&
\\
\ge\sup_{a\otimes 1+1\otimes b\le\Ga\atop a,b\in C_b(\bR^d\times\bR^d)}\left(\int a(x,\xi)f(x,\xi)dxd\xi+\int b(y,\eta)\widetilde W_\hbar[R](y,\eta)dyd\eta\right)&
\\
=\MKd(f,\widetilde W_\hbar[R])^2&\,,
\ea
$$
where the last equality follows from Kantorovich duality (Theorem 1 in chapter 1 of \cite{VillaniAMS}).

\subsection{Proof of (3)}


Since $W_\hbar[R_\hbar]\to\mu$ in $\cS'(\bR^d\times\bR^d)$, one has $\widetilde W_\hbar[R_\hbar]\to\mu$ weakly in the sense of probability measures on $\bR^d\times\bR^d$ as $\hbar\to 0$, by Theorem III.1 (1) in \cite{LionsPaul}.
Statement (2) implies that 
$$
\varliminf_{\hbar\to 0}E_\hbar(f,R_\hbar)\ge\varliminf_{\hbar\to 0}\MKd(f,\widetilde W_\hbar[R_\hbar])^2
$$
and Remark 6.12 in \cite{VillaniTOT} implies that
$$
\varliminf_{\hbar\to 0}\MKd(f,\widetilde W_\hbar[R_\hbar])^2\ge\MKd(f,\mu)^2\,.
$$


\section{Proof of Theorem \ref{T-HV}}


The proof is based on the same idea of an Eulerian variant of the Dobrushin estimate \cite{Dobru} for the mean-field limit as in \cite{FGMouPaul}.

\subsection{Step 1: Growth of the Moments of $f$.}


\begin{Lem}
Let $f^{in}$ be a probability density on $\bR^d\times\bR^d$ such that
$$
\iint_{\bR^d\times\bR^d}(|x|^2+|\xi|^2)f^{in}(x,\xi)dxd\xi<\infty\,,
$$
and let $f$ be the solution of the Cauchy problem (\ref{Vlasov}). Then, for each $t\ge 0$, one has
$$
\ba
\iint_{\bR^d\times\bR^d}\tfrac12(|x|^2+|\xi|^2)f(t,x,\xi)dxd\xi&
\\
\le e^t\left(\iint_{\bR^d\times\bR^d}\tfrac12(|x|^2+|\xi|^2)f^{in}(x,\xi)dxd\xi+\|V\|_{L^\infty}\right)&\,.
\ea
$$
\end{Lem}

\begin{proof}
First we recall the conservation of energy for solutions of the Vlasov equation:
$$
\frac{d}{dt}\left(\iint_{\bR^d\times\bR^d}\tfrac12|\xi|^2f(t,x,\xi)dxd\xi+\tfrac12\iint_{\bR^d\times\bR^d}V(x-y)\rho_f(t,x)\rho_f(t,y)dxdy\right)=0,
$$
so that
$$
\iint_{\bR^d\times\bR^d}\tfrac12|\xi|^2f(t,x,\xi)dxd\xi\le\iint_{\bR^d\times\bR^d}\tfrac12|\xi|^2f^{in}(x,\xi)dxd\xi+\|V\|_{L^\infty}
$$
for each $t\ge 0$. 

On the other hand
$$
\frac{d}{dt}\iint_{\bR^d\times\bR^d}\tfrac12|x|^2f(t,x,\xi)dxd\xi=\iint_{\bR^d\times\bR^d}\xi\cdot xf(t,x,\xi)dxd\xi
$$
so that
$$
m_2(t):=\iint_{\bR^d\times\bR^d}\tfrac12(|x|^2+|\xi|^2)f(t,x,\xi)dxd\xi
$$
satisfies
$$
m_2(t)\le m_2(0)+\|V\|_{L^\infty}+\int_0^tm_2(s)ds\,.
$$
The conclusion follows from Gronwall's lemma.
\end{proof}

\subsection{Step 2: a Dynamics for Couplings of $f$ and $R_\hbar$.}


For all $Q^{in}_\hbar\in\cC(f^{in},R^{in}_\hbar)$, let $Q_\hbar\equiv Q(t,x,\xi)$ be the solution of the Cauchy problem
\be\lb{PbCCoupl}
\left\{
\ba
{}&\d_tQ_\hbar\!+\!(\xi\cdot\grad_x\!-\!\grad V\star_x\rho_f(t,x)\cdot\grad_\xi)Q_\hbar\!+\!\left[-\tfrac12i\hbar\Dlt_y\!+\!\frac{i}{\hbar}V\star\rho[R_\hbar](t,y),Q_\hbar\right]\!=\!0\,,
\\
&Q_\hbar\rstr_{t=0}=Q_\hbar^{in}\,.
\ea
\right.
\ee

Let us briefly recall how the solution of (\ref{PbCCoupl}) is obtained. Denoting $z=(x,\xi)$, let $t\mapsto Z(t,s,z)\in\bR^d\times\bR^d$ be the integral curve of the time-dependent vector field $(\xi,-\grad V\star\rho_f(t,x))$ passing through $(x,\xi)$
at time $t=s$. Let $t\mapsto M(t,s)$ be the time-dependent unitary operator on $\fH$ such that
$$
i\hbar\d_tM_\hbar+\tfrac12\hbar^2\Dlt_yM_\hbar-V\star\rho[R_\hbar](t,y)M_\hbar=0\,,\qquad M_\hbar\rstr_{t=0}=I_\fH\,.
$$
Then
$$
Q_\hbar(t,z)=M_\hbar(t)Q^{in}_\hbar(Z(0,t,z))M_\hbar(t)^*
$$
(see \cite{BoveDPF}, especially Proposition 4.3).

\begin{Lem}
Let $Q^{in}_\hbar\in\cC(f^{in},R^{in}_\hbar)$. Then $Q_\hbar(t)\in\cC(f(t),R_\hbar(t))$ for each $t\ge 0$. 
\end{Lem}

\begin{proof}
Since $M(t)$ is unitary
$$
\ba
\Tr(Q_\hbar(t,z))&=\Tr(M_\hbar(t)Q^{in}_\hbar(Z(0,t,z))M_\hbar(t)^*)
\\
&=\Tr(M_\hbar(t)^*M_\hbar(t)Q^{in}_\hbar(Z(0,t,z)))
\\
&=\Tr(Q^{in}_\hbar(Z(0,t,z)))=f^{in}(Z(0,t,z))=f(t,z)\,,
\ea
$$
while
$$
\ba
\int_{\bR^d\times\bR^d}Q_\hbar(t,z)dxd\xi&=M_\hbar(t)\left(\int_{\bR^d\times\bR^d}Q^{in}_\hbar(Z(0,t,z))dxd\xi\right)M_\hbar(t)^*
\\
&=M_\hbar(t)R_\hbar^{in}M_\hbar(t)^*=R_\hbar(t)\,,
\ea
$$
since the measure $dxd\xi$ is invariant under the Hamiltonian flow $Z(s,t,\cdot)$ for all $s,t\in\bR$.
\end{proof}

\subsection{Step 3: the Eulerian Estimate for the Second Order Moment of $Q_\hbar$.}


We define
$$
\cE_\hbar(t):=\int_{\bR^d\times\bR^d}\Tr(c_\hbar(z)Q_\hbar(t,z))dxd\xi\,.
$$
Since $Q_\hbar$ is a solution of (\ref{PbCCoupl}), one has
$$
\ba
\frac{d\cE_\hbar}{dt}(t)=&\int_{\bR^d\times\bR^d}\Tr(Q_\hbar(t,z)\{\tfrac12|\xi|^2+V\star\rho_f(t,x),c_\hbar(z)\})dxd\xi
\\
&+\int_{\bR^d\times\bR^d}\Tr\left(Q_\hbar(t,z)\left[-\tfrac12i\hbar\Dlt_y+\frac{i}{\hbar}V\star\rho[R_\hbar](t,y),c_\hbar(z)\right]\right)dxd\xi\,.
\ea
$$

We shall need the following auxiliary computations. First
$$
\{\tfrac12|\xi|^2,c_\hbar(z)\}=\xi\cdot(x-y)\,,
$$
while
$$
\ba
{}[-\tfrac12i\hbar\Dlt_y,c_\hbar]=&-\tfrac12i\hbar[\grad_y,c_\hbar]\cdot\grad_y-\tfrac12i\hbar\grad_y\cdot[\grad_y,c_\hbar]
\\
=&-\tfrac12i\hbar(y-x)\cdot\grad_y-\tfrac12i\hbar\grad_y\cdot(y-x)\,,
\ea
$$
so that
$$
\{\tfrac12|\xi|^2,c_\hbar(z)\}-\left[\tfrac12i\hbar\Dlt_y,c_\hbar(z)\right]=\tfrac12(x-y)\cdot(\xi+i\hbar\grad_y)+\tfrac12(\xi+i\hbar\grad_y)\cdot(x-y)\,.
$$

Likewise
$$
\{V\star\rho_f(t,x),c_\hbar(z)\}=-\grad V\star\rho_f(t,x)\cdot(\xi+i\hbar\grad_y)\,,
$$
while
$$
\ba
\left[c_\hbar(z),\frac{i}{\hbar}V\star\rho[R_\hbar](t,y)\right]=&-\tfrac12\grad V\star\rho[R_\hbar](t,y)\cdot(\xi+i\hbar\grad_y)
\\
&-\tfrac12(\xi+i\hbar\grad_y)\cdot\grad V\star\rho[R_\hbar](t,y)\,,
\ea
$$
so that
$$
\ba
\{V\star\rho_f(t,x),c_\hbar(z)\}+\frac{i}{\hbar}\left[V\star\rho[R_\hbar](t,y),c_\hbar(z),\right]&
\\
=\tfrac12(\grad V\star\rho[R_\hbar](t,y)-\grad V\star\rho_f(t,x))\cdot(\xi+i\hbar\grad_y)&
\\
+\tfrac12(\xi+i\hbar\grad_y)\cdot(\grad V\star\rho[R_\hbar](t,y)-\grad V\star\rho_f(t,x))&\,.
\ea
$$

Hence
$$
\ba
\frac{d\cE_\hbar}{dt}(t)\!=\!\int_{\bR^d\times\bR^d}\Tr(Q_\hbar(t,z)\tfrac12((x-y)\!\cdot\!(\xi\!+\!i\hbar\grad_y)\!+\!(\xi\!+\!i\hbar\grad_y)\!\cdot\!(x-y)))dxd\xi&
\\
+\int_{\bR^d\times\bR^d}\Tr(Q_\hbar(t,z)\tfrac12(\grad V\star\rho[R_\hbar](t,y)-\grad V\star\rho_f(t,x))\cdot(\xi+i\hbar\grad_y))dxd\xi&
\\
+\int_{\bR^d\times\bR^d}\Tr(Q_\hbar(t,z)\tfrac12(\xi+i\hbar\grad_y)\cdot(\grad V\star\rho[R_\hbar](t,y)-\grad V\star\rho_f(t,x)))dxd\xi&\,.
\ea
$$

At this point, we recall that, if $A,B$ are self-adjoint, possibly unbounded operators on $\fH$, and if $R\in\cL(\fH)$ satisfies $R=R^*\ge 0$, then
\be\lb{NCCS}
\Tr(R(AB+BA))\le\Tr(R(A^2+B^2))\,,
\ee
since
$$
\Tr(R(A-B)^2)\ge 0\,.
$$
(If $T$ is an unbounded, self-adjoint nonnegative operator on $\fH$, we define $\Tr(R T)$ as an element of $[0,+\infty]$ by the formula $\Tr(R T):=\Tr(R^{1/2}TR^{1/2})$, even if $RT$ is not a trace-class operator on $\fH$.)

Hence
$$
\int_{\bR^d\times\bR^d}\Tr(Q_\hbar(t,z)\tfrac12((x-y)\cdot(\xi+i\hbar\grad_y)+(\xi+i\hbar\grad_y)\cdot(x-y)))dxd\xi\le\cE_\hbar(t)\,,
$$
and
$$
\ba
\int_{\bR^d\times\bR^d}\Tr(Q_\hbar(t,z)\tfrac12(\grad V\star\rho[R_\hbar](t,y)-\grad V\star\rho_f(t,x))\cdot(\xi+i\hbar\grad_y))dxd\xi&
\\
+\int_{\bR^d\times\bR^d}\Tr(Q_\hbar(t,z)\tfrac12(\xi+i\hbar\grad_y)\cdot(\grad V\star\rho[R_\hbar](t,y)-\grad V\star\rho_f(t,x)))dxd\xi&
\\
\le\int_{\bR^d\times\bR^d}\Tr(Q_\hbar(t,z)\tfrac12|\grad V\star\rho_f(t,x)-\grad V\star\rho[R_\hbar](t,y)|^2)dxd\xi&
\\
+\int_{\bR^d\times\bR^d}\Tr(Q_\hbar(t,z)\tfrac12|\xi+i\hbar\grad_y|^2)dxd\xi&\,.
\ea
$$
Next
$$
\ba
\int_{\bR^d\times\bR^d}\Tr(Q_\hbar(t,z)\tfrac12|\grad V\star\rho_f(t,x)-\grad V\star\rho[R_\hbar](t,y)|^2)dxd\xi&
\\
\le\int_{\bR^d\times\bR^d}\Tr(Q_\hbar(t,z)|\grad V\star\rho_f(t,x)-\grad V\star\rho[R_\hbar](t,x)|^2)dxd\xi&
\\
+\int_{\bR^d\times\bR^d}\Tr(Q_\hbar(t,z)|\grad V\star\rho[R_\hbar](t,x)-\grad V\star\rho[R_\hbar](t,y)|^2)dxd\xi&\,.
\ea
$$
Observe that
$$
\ba
|\grad V\star\rho[R_\hbar](t,x)-\grad V\star\rho[R_\hbar](t,y)|
\\
=\left|\int_{\bR^d}(\grad V(x-z)-\grad V(y-z))\rho[R_\hbar](t,z)dz\right|
\\
=\int_{\bR^d}|\grad V(x-z)-\grad V(y-z)|\rho[R_\hbar](t,z)dz
\\
\le\Lip(\grad V)|x-y|\int_{\bR^d}\rho[R_\hbar](t,z)dz
\\
\le\Lip(\grad V)|x-y|\,.
\ea
$$
Hence
$$
\ba
\int_{\bR^d\times\bR^d}\Tr(Q_\hbar(t,z)|\grad V\star\rho[R_\hbar](t,x)-\grad V\star\rho[R_\hbar](t,y)|^2)dxd\xi
\\
\le\Lip(\grad V)^2\int_{\bR^d\times\bR^d}\Tr(Q_\hbar(t,z)|x-y|^2)dxd\xi\,.
\ea
$$

Summarizing, we have proved that
$$
\ba
\frac{d\cE_\hbar}{dt}(t)\le\cE_\hbar(t)+\int_{\bR^d\times\bR^d}|\grad V\star\rho_f(t,x)-\grad V\star\rho[R_\hbar](t,x)|^2f(t,x,\xi)dxd\xi
\\
+\int_{\bR^d\times\bR^d}\Tr(Q_\hbar(t,z)(\tfrac12|\xi+i\hbar\grad_y|^2+\Lip(\grad V)^2|x-y|^2))dxd\xi&\,.
\ea
$$

Observe that
$$
\ba
\grad V\star\rho_f(t,X)-\grad V\star\rho[R_\hbar](t,X)
\\
=\int_{\bR^d\times\bR^d}\grad V(X-x)f(t,x,\xi)dxd\xi-\int_{\bR^d}\grad V(X-y)\rho[R_\hbar](t,y)dy
\\
=\int_{\bR^d\times\bR^d}\grad V(X-x)f(t,x,\xi)dxd\xi-\Tr(\grad V(X-\cdot)R_\hbar)
\\
=\int_{\bR^d\times\bR^d}\Tr((\grad V(X-x)-\grad V(X-y))Q_\hbar(t,x,\xi))dxd\xi
\ea
$$
so that, by the Cauchy-Schwarz inequality, 
$$
\ba
|\grad V\star\rho_f(t,X)-\grad V\star\rho[R_\hbar](t,X)|^2&
\\
\le\int_{\bR^d\times\bR^d}\Tr(|\grad V(X-x)-\grad V(X-y)|^2Q_\hbar(t,x,\xi))dxd\xi&
\\
\le\Lip(\grad V)^2\int_{\bR^d\times\bR^d}\Tr(|x-y|^2Q_\hbar(t,x,\xi))dxd\xi&\,.
\ea
$$

Hence
$$
\ba
\frac{d\cE_\hbar}{dt}(t)\le\cE_\hbar(t)+\Lip(\grad V)^2\int_{\bR^d\times\bR^d}\Tr(|x-y|^2Q_\hbar(t,x,\xi))dxd\xi
\\
+\int_{\bR^d\times\bR^d}\Tr(Q_\hbar(t,z)(\tfrac12|\xi+i\hbar\grad_y|^2+\Lip(\grad V)^2|x-y|^2))dxd\xi&\,,
\ea
$$
or, in other words,
$$
\frac{d\cE_\hbar}{dt}(t)\le(1+\max(4\Lip(\grad V)^2,1))\cE_\hbar(t)\,.
$$
Therefore
$$
E_\hbar(f(t),R_\hbar(t))^2=\cE_\hbar(t)\le e^{\Lambda t}\cE_\hbar(0)\,,\quad\hbox{ for all }t\ge 0\,.
$$

Minimizing the left hand side of the inequality above as $Q_\hbar^{in}$ runs through $\cC(f^{in},R_\hbar^{in})$, we conclude that
$$
E_\hbar(f(t),R_\hbar(t))^2\le e^{\Lambda t}E_\hbar(f^{in},R_\hbar^{in})^2\,,\quad\hbox{ for all }t\ge 0\,.
$$

\subsection{Step 4: Convergence Rate and Monge-Kantorovich Distances.}


First, we apply statement (2) in Theorem \ref{T-PtyE} to the left-hand side of the inequality above. One finds that
$$
\ba
\MKd(f(t),\widetilde W_\hbar[R_\hbar(t)])^2\le&E_\hbar(f(t),R_\hbar(t))^2+\tfrac12d\hbar
\\
\le&e^{\Lambda t}E_\hbar(f^{in},R_\hbar^{in})^2+\tfrac12d\hbar\,,\quad\hbox{ for all }t\ge 0\,.
\ea
$$
In addition, if $R_\hbar^{in}=\Op^T_\hbar((2\pi\hbar)^d\mu^{in})$ where $\mu^{in}$ is a Borel probability measure on $\bR^d$, we apply statement (3) in Theorem \ref{T-PtyE} to the right-hand side of the previous inequality, to find that
$$
\MKd(f(t),\widetilde W_\hbar[R_\hbar(t)])^2\le e^{\Lambda t}\left(\MKd(f^{in},\mu^{in})^2+\tfrac12d\hbar\right)+\tfrac12d\hbar
$$
for all $t\ge 0$.


\section{Proof of Theorem \ref{T-NSV}}


\subsection{Step 1: a Dynamics for Couplings of $f^{\otimes N}$ and $R_{\hbar,N}$.}


Consider an arbitrary coupling $Q^{in}_{\hbar,N}\in\cC((f^{in})^{\otimes N},R^{in}_{\hbar,N})$, satisfying the symmetry condition
\be\lb{NSym}
Q^{in}_{\hbar,N}(\si\cdot X_N,\si\cdot\Xi_N)=U^*_\si Q^{in}_{\hbar,N}(X_N,\Xi_N)U_\si\,.
\ee
The set of all such couplings is denoted $\cC^s((f^{in})^{\otimes N},R^{in}_{\hbar,N})$.

Let $Q_{\hbar,N}\equiv Q_{\hbar,N}(t,X_N,\Xi_N)$ be the solution of the Cauchy problem
\be\lb{EqNCoupl}
\ba
\d_tQ_{\hbar,N}(t,X_N,\Xi_N)&+\sum_{j=1}^N(\xi_j\cdot\grad_{x_j}-\grad V\star_x\rho_f(t,x_j)\cdot\grad_{\xi_j})Q_{\hbar,N}(t,X_N,\Xi_N)
\\
&+\frac{i}{\hbar}[\cH_{\hbar,N},Q_{\hbar,N}(t,X_N,\Xi_N)]_N=0\,,
\ea
\ee
with initial data 
\be\lb{InNcoupl}
Q_{\hbar,N}(0,X_N,\Xi_N)=Q^{in}_{\hbar,N}(X_N,\Xi_N)\,.
\ee
Here, we denote by $[\cdot,\cdot]_N$ the commutator between operators on $\fH_N$, and we recall that the quantum Hamiltonian $\cH_{\hbar,N}$ is defined in formula (\ref{QHamN}).

\begin{Lem}\lb{L-NQCoupl}
Let $Q_{\hbar,N}$ be a solution of (\ref{EqNCoupl}) with initial data (\ref{InNcoupl}) satisfying the symmetry (\ref{NSym}). Then

\smallskip
\noindent
(a) one has 
$$
Q_{\hbar,N}(t)\in\cC(f(t)^{\otimes N},R_{\hbar,N}(t))\quad\hbox{ for each }t\ge 0\,;
$$
(b) for each $\si\in\fS_N$ and each $t\ge 0$, one has
$$
Q_{\hbar,N}(t,\si\cdot X_N,\si\cdot\Xi_N)=U^*_\si Q_{\hbar,N}(t,X_N,\Xi_N)U_\si\,.
$$
In other words, $Q_{\hbar,N}(t)\in\cC^s(f(t)^{\otimes N},R_{\hbar,N}(t))$ for all $t\ge 0$.
\end{Lem}

\smallskip
In particular, statement (a) implies that $\Tr(Q_{\hbar,N}(t))=f(t)^{\otimes N}$. This factorization property is of considerable importance for the proof of Theorem \ref{T-NSV} --- in particular for the consistency part of Step 4 below.

\begin{proof}
For each $(X_N,\Xi_N)\in(\bR^d)^N\times(\bR^d)^N$, let $t\mapsto Z_N(t,s,X_N,\Xi_N)$ be the integral curve of the time-dependent vector field
\be\lb{NVectField}
(\xi_1,\ldots,\xi_N,-\grad V\star_x\rho_f(t,x_1),\ldots,-\grad V\star_x\rho_f(t,x_N))
\ee
passing through $(X_N,\Xi_N)$ at time $t=s$. Then 
$$
Q_{\hbar,N}(t,X_N,\Xi_N)=e^{-it\cH_{\hbar,N}/\hbar}Q^{in}_{\hbar,N}(Z_N(0,t,X_N,\Xi_N))e^{it\cH_{\hbar,N}/\hbar}\,.
$$
Thus
$$
\ba
\Tr(Q_{\hbar,N}(t,X_N,\Xi_N))&=\Tr(Q^{in}_{\hbar,N}(Z_N(0,t,X_N,\Xi_N))
\\
&=(f^{in})^{\otimes N}(Z_N(0,t,X_N,\Xi_N))=\prod_{j=1}^Nf(t,x_j,\xi_j)\,,
\ea
$$
while
$$
\ba
\int_{(\bR^d)^N\times(\bR^d)^N}Q_{\hbar,N}(t,X_N,\Xi_N)dX_Nd\Xi_N
\\
=e^{-it\cH_{\hbar,N}/\hbar}\left(\int_{(\bR^d)^N\times(\bR^d)^N}Q^{in}_{\hbar,N}(Y_N,H_N)dY_NdH_N\right)e^{it\cH_{\hbar,N}/\hbar}
\\
=e^{-it\cH_{\hbar,N}/\hbar}R^{in}_{\hbar,N}e^{it\cH_{\hbar,N}/\hbar}=R_{\hbar,N}(t)\,,
\ea
$$
since the flow $Z_N$ leaves the Lebesgue measure of $(\bR^d)^N\times(\bR^d)^N$ invariant. This proves (a).

As for (b), observe that $\tilde Q_{\hbar,N}(t,X_N,\Xi_N):=U_\si Q_{\hbar,N}(t,\si\cdot X_N,\si\cdot\Xi_N)U^*_\si$ satisfies (\ref{EqNCoupl}), because $U_\si\cH_{\hbar,N}=\cH_{\hbar,N}U_\si$ while the vector field (\ref{NVectField}) is invariant 
under the transformation $(X_N,\Xi_N)\mapsto(\si\cdot X_N,\si\cdot\Xi_N)$. On the other hand $\tilde Q_{\hbar,N}(0)=Q^{in}_{\hbar,N}$ according to (\ref{NSym}). By uniqueness of the solution of the Cauchy problem for (\ref{EqNCoupl}), one 
has $\tilde Q_{\hbar,N}(t)=Q_{\hbar,N}(t)$ for all $t\ge 0$. This identity obviously holds for each $\si\in\fS_N$.
\end{proof}

\subsection{Step 2: Coupling BBGKY Hierarchies.}


Instead of working directly with the equation (\ref{EqNCoupl}) for $N$-particle couplings, we look at the hierarchy of equations for the $n$-particle marginals of $Q_{\hbar,N}$. 

\begin{Def}\label{defmargi}
For each $n=1,\ldots,N$, we define the $n$-particle marginal
of $Q_{\hbar,N}$, henceforth denoted $Q_{\hbar,N}^\mathbf{n}$, as follows: for a.e. $X_n,\Xi_n\in(\bR^d)^n$,
$$
Q_{\hbar,N}^\mathbf{n}(t,X_n,\Xi_n):=\int_{(\bR^d\times\bR^d)^{N-n}}[Q_{\hbar,N}(t,X_N,\Xi_N)]^\mathbf{n}dx_{n+1}d\xi_{n+1}\ldots dx_Nd\xi_N\,,
$$
where the $n$-particle marginal of the $N$-particle density $Q_{\hbar,N}(t,X_N,\Xi_N)$ has been defined in (\ref{nQMargi}).
\end{Def}

Integrating with respect to $x_2,\xi_2,\ldots,x_N,\xi_N$ and taking the $1$st particle marginal of both sides of (\ref{EqNCoupl}) leads to the equation
$$
\ba
\d_tQ^\mathbf{1}_{\hbar,N}(t,x_1,\xi_1)+(\xi_1\cdot\grad_{x_1}-\grad V\star_x\rho_f(t,x_1)\cdot\grad_{\xi_1})Q^\mathbf{1}_{\hbar,N}(t,x_1,\xi_1)&
\\
+\frac{i}{\hbar}\int_{(\bR^d\times\bR^d)^{N-1}}[\cH_{\hbar,N},Q_{\hbar,N}(t,X_N,\Xi_N)]^\mathbf{1}_Ndx_2d\xi_2\ldots,dx_Nd\xi_N=0&\,.
\ea
$$
Then
$$
\ba
\left[\cH_{\hbar,N},Q_{\hbar,N}(t,X_N,\Xi_N)\right]^\mathbf{1}_N=&[-\tfrac12\hbar^2\Dlt_{y_1},Q_{\hbar,N}(t,X_N,\Xi_N)^\mathbf{1}]_1
\\
&+\frac1N\sum_{j=2}^N[V(y_1-y_j),Q_{\hbar,N}(t,X_N,\Xi_N)]_N^\mathbf{1}\,.
\ea
$$
At this point, we integrate further in $x_2,\xi_2,\ldots,x_N,\xi_N$, and find that
$$
\ba
\int_{(\bR^d\times\bR^d)^{N-1}}\left[\cH_{\hbar,N},Q_{\hbar,N}(t,X_N,\Xi_N)\right]^\mathbf{1}_Ndx_2d\xi_2\ldots,dx_Nd\xi_N
\\
=\left[-\tfrac12\hbar^2\Dlt_{y_1},\int_{(\bR^d\times\bR^d)^{N-1}}Q_{\hbar,N}(t,X_N,\Xi_N)^\mathbf{1}dx_2d\xi_2\ldots,dx_Nd\xi_N\right]_1
\\
+\frac1N\sum_{j=2}^N\int_{(\bR^d\times\bR^d)^{N-1}}[V(y_1-y_j),Q_{\hbar,N}(t,X_N,\Xi_N)]_N^\mathbf{1}dx_2d\xi_2\ldots,dx_Nd\xi_N\,.
\ea
$$
Using the symmetry relation (b) in Lemma \ref{L-NQCoupl}, we observe that
$$
[V(y_1-y_j),Q_{\hbar,N}(t,X_N,\Xi_N)]_N^\mathbf{1}=[V(y_1-y_2),Q_{\hbar,N}(t,\si\cdot X_N,\si\cdot\Xi_N)]_N^\mathbf{1}
$$
where $\si$ is the permutation exchanging $2$ and $j$ and leaving all the other indices invariant. Hence
$$
\ba
\int_{(\bR^d\times\bR^d)^{N-1}}&\left[\cH_{\hbar,N},Q_{\hbar,N}(t,X_N,\Xi_N)\right]^\mathbf{1}_Ndx_2d\xi_2\ldots,dx_Nd\xi_N
\\
&=[-\tfrac12\hbar^2\Dlt_{y_1},Q^\mathbf{1}_{\hbar,N}(t,x_1,\xi_1)]_1
\\
&+\frac{N-1}N\int_{\bR^d\times\bR^d}[V(y_1-y_2),Q^\mathbf{2}_{\hbar,N}(t,x_1,x_2,\xi_1,\xi_2)]_2^\mathbf{1}dx_2d\xi_2\,.
\ea
$$

Eventually, we arrive at the equation for $Q_{\hbar,N}^\mathbf{1}$, which is
\be\lb{BBGKY1}
\ba
\d_tQ^\mathbf{1}_{\hbar,N}(t,x_1,\xi_1)&
\\
+(\xi_1\cdot\grad_{x_1}-\grad V\star_x\rho_f(t,x_1)\cdot\grad_{\xi_1})Q^\mathbf{1}_{\hbar,N}(t,x_1,\xi_1)-\tfrac12i\hbar[\Dlt_{y_1},Q^\mathbf{1}_{\hbar,N}(t,x_1,\xi_1)]_1&
\\
+\frac{i}{\hbar}\int_{\bR^d\times\bR^d}[\tfrac{N-1}NV(y_1-y_2),Q^\mathbf{2}_{\hbar,N}(t,x_1,x_2,\xi_1,\xi_2)]_2^\mathbf{1}dx_2d\xi_2=0&\,.
\ea
\ee
This equation couples the first particle marginals of $Q_{\hbar,N}$, and is not in closed form, as it involves the $2$-particle marginal $Q^\mathbf{2}_{\hbar,N}$. Therefore, this equation can be regarded as coupling the first equation
in the BBGKY hierarchy for the $N$-particle Schr\"odinger equation with the Vlasov equation. 

\begin{Rmk}
Proceeding in the same way, one could also write the analogous equation coupling the $n$-th equation in the BBGKY hierarchy for the $N$-particle Schr\"odinger equation with the equation satisfied by the $n$-fold tensor product 
of the Vlasov solution:
\be\lb{BBGKYn}
\ba
\d_tQ^\mathbf{n}_{\hbar,N}(t,X_n,\Xi_n)+\sum_{j=1}^n(\xi_j\cdot\grad_{x_j}-\grad V\star_x\rho_f(t,x_j)\cdot\grad_{\xi_j})Q^\mathbf{n}_{\hbar,N}(t,X_n,\Xi_n)
\\
+\left[-\tfrac12i\hbar\sum_{j=1}^n\Dlt_{y_j}+\frac{i}{2N\hbar}\sum_{j,k=1}^nV(y_j-y_k),Q^\mathbf{n}_{\hbar,N}(t,X_n,\Xi_n)\right]_n&
\\
+\frac{i}{\hbar}\sum_{j=1}^n\int_{\bR^d\times\bR^d}[\tfrac{N-n}NV(y_j-y_{n+1}),Q^\mathbf{n+1}_{\hbar,N}(t,X_{n+1},\Xi_{n+1})]_{n+1}^\mathbf{n}dx_{n+1}d\xi_{n+1}=0&\,.
\ea
\ee
For $n=N$, this equation obviously coincides with the original equation (\ref{EqNCoupl}), with the convention $Q_{\hbar,N}^\mathbf{N+1}:=0$. Analogously, we recall that the last equation in the BBGKY equation coincides with the 
$N$-particle Liouville, or von Neumann equation. 
\end{Rmk}

In the course of the proof, we shall use \textit{only the first }equation (\ref{BBGKY1}) in the hierarchy of equations (\ref{BBGKYn}) for $n=1,\ldots,N-1$, together with the original equation (\ref{EqNCoupl}). Equations (\ref{BBGKY1}) and
(\ref{EqNCoupl}) are used \textit{for two different purposes} in the proof. We shall return to this later.

\subsection{Step 3: On Moments and Symmetries of $Q_{\hbar,N}(t)$.}


Henceforth we denote
$$
c_{\hbar,j}(x,\xi):=\tfrac12(|x-y_j|^2+|\xi+i\hbar\grad_{y_j}|^2)\,,\quad j=1,\ldots,N\,.
$$
Define
$$
\cD_{\hbar,N}(t):=\int_{(\bR^d\times\bR^d)^N}\frac1N\sum_{j=1}^N\Tr_{\fH_N}(c_{\hbar,j}(x_j,\xi_j)Q_{\hbar,N}(t,X_N,\Xi_N))dX_Nd\Xi_N\,.
$$

\begin{Lem}\lb{L-DMargi}
For each $n=1,\ldots,N$ and each $t\ge 0$, one has
$$
\cD_{\hbar,N}(t)=\int_{(\bR^d\times\bR^d)^n}\frac1n\sum_{j=1}^n\Tr_{\fH_n}(c_{\hbar,j}(x_j,\xi_j)Q^\mathbf{n}_{\hbar,N}(t,X_n,\Xi_n))dX_nd\Xi_n
$$
\end{Lem}

\begin{proof}
By statement (b) in Lemma \ref{L-NQCoupl}, one has
$$
\ba
\int_{(\bR^d\times\bR^d)^N}\Tr_{\fH_N}(c_{\hbar,j}(x_j,\xi_j)Q_{\hbar,N}(t,X_N,\Xi_N))dX_Nd\Xi_N
\\
=\int_{(\bR^d\times\bR^d)^N}\Tr_{\fH_N}(U^*_\si c_{\hbar,1}(x_j,\xi_j)U_\si Q_{\hbar,N}(t,X_N,\Xi_N))dX_Nd\Xi_N
\\
=\int_{(\bR^d\times\bR^d)^N}\Tr_{\fH_N}(c_{\hbar,1}(x_1,\xi_1)U_\si Q_{\hbar,N}(t,\si\cdot X_N,\si\cdot\Xi_N)U^*_\si)dX_Nd\Xi_N
\\
=\int_{(\bR^d\times\bR^d)^N}\Tr_{\fH_N}(c_{\hbar,1}(x_1,\xi_1)Q_{\hbar,N}(t,X_N,\Xi_N))dX_Nd\Xi_N
\ea
$$
if $\si$ is the permutation exchanging the indices $1$ and $j$, and leaving the other indices invariant. Therefore
$$
\ba
\cD_{\hbar,N}(t)=\int_{(\bR^d\times\bR^d)^N}\Tr_{\fH_N}(c_{\hbar,1}(x_1,\xi_1)Q_{\hbar,N}(t,X_N,\Xi_N))dX_Nd\Xi_N&
\\
=\int_{\bR^d\times\bR^d}\Tr_{\fH}(c_{\hbar,1}(x_1,\xi_1)Q^\mathbf{1}_{\hbar,N}(t,x_1,\xi_1))dx_1d\xi_1&\,,
\ea
$$
and by the same token
$$
\ba
\cD_{\hbar,N}(t)&=\int_{(\bR^d\times\bR^d)^N}\frac1n\sum_{j=1}^n\Tr_{\fH_N}(c_{\hbar,j}(x_j,\xi_j)Q_{\hbar,N}(t,X_N,\Xi_N))dX_Nd\Xi_N
\\
&=\int_{(\bR^d\times\bR^d)^n}\frac1n\sum_{j=1}^n\Tr_{\fH_n}(c_{\hbar,j}(x_j,\xi_j)Q^\mathbf{n}_{\hbar,N}(t,X_n,\Xi_n))dX_nd\Xi_n
\ea
$$
for all $n=1,\ldots,N$.
\end{proof}

\subsection{Step 4: a Differential Inequality for $\cD_{\hbar,N}(t)$.}


The most important step in the proof of Theorem \ref{T-NSV} is to obtain a differential inequality controlling the convergence rate of $R_{\hbar,N}$ to $f$ in some appropriate sense. This control involves two different arguments, referred to
as ``stability'' and ``consistency'' by analogy with the Lax equivalence theorem \cite{LaxRicht} in numerical analysis.

\subsubsection{Stability}


Multiplying both sides of (\ref{BBGKY1}) by $c_{\hbar,1}(x_1,\xi_1)$, integrating in $x_1,\xi_1$ and taking the trace of both sides of the resulting equality, we arrive at the following identity:
$$
\ba
\frac{d \cD_{\hbar,N}}{dt}(t)
\\
\!=\!\int_{\bR^d\times\bR^d}\!\!\!\Tr_{\fH}(Q_{\hbar,N}^\mathbf{1}(t,x_1,\xi_1)(\xi_1\!\cdot\!\grad_{x_1}c_{\hbar,1}(x_1,\xi_1)\!-\!\tfrac12i\hbar[\Dlt_{y_1},c_{\hbar,1}(x_1,\xi_1)])dx_1d\xi_1
\\
-\int_{\bR^d\times\bR^d}\!\Tr_{\fH}(Q_{\hbar,N}^\mathbf{1}(t,x_1,\xi_1)\grad V\star_x\rho_f(t,x_1)\cdot\grad_{\xi_1}c_{\hbar,1}(x_1,\xi_1))dx_1d\xi_1
\\
+\frac{i}{\hbar}\frac{N-1}N\int_{\bR^d\times\bR^d}\Tr(Q^\mathbf{2}_{\hbar,N}(t,X_2,\Xi_2)[V(y_1-y_2),c_{\hbar,1}(x_1,\xi_1)]_2)dx_2d\xi_2&\,.
\ea
$$
We shall use the following auxiliary computations:
$$
\xi_1\cdot\grad_{x_1}c_{\hbar,1}(x_1,\xi_1)=\xi_1\cdot\grad_{x_1}\tfrac12|x_1-y_1|^2=\xi_1\cdot(x_1-y_1)\,,
$$
while
$$
\ba
-\tfrac12i\hbar[\Dlt_{y_1},c_{\hbar,1}(x_1,\xi_1)]=&-\tfrac12i\hbar\grad_{y_1}\cdot[\grad_{y_1},\tfrac12|x_1-y_1|^2]
\\
&-\tfrac12i\hbar[\grad_{y_1},\tfrac12|x_1-y_1|^2]\cdot\grad_{y_1}
\\
=&-\tfrac12i\hbar\grad_{y_1}\cdot(y_1-x_1)-\tfrac12i\hbar(y_1-x_1)\cdot\grad_{y_1}
\ea
$$
so that
$$
\ba
(\xi_1\cdot\grad_{x_1}c_{\hbar,1}(x_1,\xi_1)-\tfrac12i\hbar[\Dlt_{y_1},c_{\hbar,1}(x_1,\xi_1)]
\\
=\tfrac12(\xi_1+i\hbar\grad_{y_1})\cdot(x_1-y_1)+\tfrac12(x_1-y_1)\cdot(\xi_1+i\hbar\grad_{y_1})&\,.
\ea
$$
Likewise
$$
\ba
-\grad V\star_x\rho_f(t,x_1)\cdot\grad_{\xi_1}c_{\hbar,1}(x_1,\xi_1)=-\grad V\star_x\rho_f(t,x_1)\cdot\grad_{\xi_1}\tfrac12|\xi_1+i\hbar\grad_{y_1}|^2
\\
=-\grad V\star_x\rho_f(t,x_1)\cdot(\xi_1+i\hbar\grad_{y_1})&\,,
\ea
$$
while
$$
\ba
\frac{i}{\hbar}[V(y_1-y_2),c_{\hbar,1}(x_1,\xi_1)]_2=\frac{i}{\hbar}[V(y_1-y_2),\tfrac12|\xi_1+i\hbar\grad_{y_1}|^2]_2
\\
=-\frac{i}{2\hbar}(\xi_1+i\hbar\grad_{y_1})\cdot[(\xi_1+i\hbar\grad_{y_1}),V(y_1-y_2)]_2
\\
-\frac{i}{2\hbar}[(\xi_1+i\hbar\grad_{y_1}),V(y_1-y_2)]_2\cdot(\xi_1+i\hbar\grad_{y_1})
\\
=\tfrac12(\xi_1+i\hbar\grad_{y_1})\cdot\grad V(y_1-y_2)+\tfrac12\grad V(y_1-y_2)\cdot(\xi_1+i\hbar\grad_{y_1})&\,.
\ea
$$

Hence
$$
\ba
\frac{d \cD_{\hbar,N}}{dt}(t)=\tfrac12\int_{\bR^d\times\bR^d}\Tr_{\fH}(Q_{\hbar,N}^\mathbf{1}(t,x_1,\xi_1)(\xi_1+i\hbar\grad_{y_1})\cdot(x_1-y_1))dx_1d\xi_1
\\
+\tfrac12\int_{\bR^d\times\bR^d}\Tr_{\fH}(Q_{\hbar,N}^\mathbf{1}(t,x_1,\xi_1)(x_1-y_1)\cdot(\xi_1+i\hbar\grad_{y_1}))dx_1d\xi_1
\\
-\int_{\bR^d\times\bR^d}\Tr_{\fH}(Q_{\hbar,N}^\mathbf{1}(t,x_1,\xi_1)\grad V\star_x\rho_f(t,x_1)\cdot(\xi_1+i\hbar\grad_{y_1}))dx_1d\xi_1
\\
+\frac{N-1}{2N}\int_{(\bR^d\times\bR^d)^2}\Tr_{\fH_2}(Q^\mathbf{2}_{\hbar,N}(t,X_2,\Xi_2)(\xi_1+i\hbar\grad_{y_1})\cdot\grad V(y_1-y_2))dX_2d\Xi_2
\\
+\frac{N-1}{2N}\int_{(\bR^d\times\bR^d)^2}\Tr_{\fH_2}(Q^\mathbf{2}_{\hbar,N}(t,X_2,\Xi_2)\grad V(y_1-y_2)\cdot(\xi_1+i\hbar\grad_{y_1}))dX_2d\Xi_2&\,.
\ea
$$

Using the inequality (\ref{NCCS}) with $A=x_1-y_1$ and $B=\xi_1+i\hbar\grad_{y_1}$ shows that
$$
\ba
\frac{d \cD_{\hbar,N}}{dt}(t)\le \cD_{\hbar,N}(t)
\\
-\int_{(\bR^d\times\bR^d)^2}\Tr_{\fH_2}(Q_{\hbar,N}^\mathbf{2}(t,X_2,\Xi_2)\cV(t,x_1,x_2)\cdot(\xi_1+i\hbar\grad_{y_1}))dX_2d\Xi_2
\\
-\frac{N-1}{2N}\int_{(\bR^d\times\bR^d)^2}\Tr_{\fH_2}(Q^\mathbf{2}_{\hbar,N}(t,X_2,\Xi_2)(\xi_1+i\hbar\grad_{y_1})\cdot\cW(X_2,Y_2))dX_2d\Xi_2
\\
-\frac{N-1}{2N}\int_{(\bR^d\times\bR^d)^2}\Tr_{\fH_2}(Q^\mathbf{2}_{\hbar,N}(t,X_2,\Xi_2)\cW(X_2,Y_2)\cdot(\xi_1+i\hbar\grad_{y_1}))dX_2d\Xi_2&\,,
\ea
$$
with the notation
$$
\cV(t,x_1,x_2):=\grad V\star_x\rho_f(t,x_1)-\tfrac{N-1}N\grad V(x_1-x_2)
$$
and
$$
\cW(X_2,Y_2):=\grad V(x_1-x_2)-\grad V(y_1-y_2)\,.
$$
Using again inequality (\ref{NCCS}) with $A=\cW(X_2,Y_2)$ and $B=\xi_1+i\hbar\grad_{y_1}$ shows that
$$
\ba
\frac{d \cD_{\hbar,N}}{dt}(t)\le \cD_{\hbar,N}(t)
\\
-\int_{(\bR^d\times\bR^d)^2}\Tr_{\fH_2}(Q_{\hbar,N}^\mathbf{2}(t,X_2,\Xi_2)\cV(t,x_1,x_2)\cdot(\xi_1+i\hbar\grad_{y_1}))dX_2d\Xi_2
\\
+\frac{N-1}{2N}\int_{(\bR^d\times\bR^d)^2}\Tr_{\fH_2}(Q^\mathbf{2}_{\hbar,N}(t,X_2,\Xi_2)|\xi_1+i\hbar\grad_{y_1}|^2)dX_2d\Xi_2
\\
+\frac{N-1}{2N}\int_{(\bR^d\times\bR^d)^2}\Tr_{\fH_2}(Q^\mathbf{2}_{\hbar,N}(t,X_2,\Xi_2)|\cW(X_2,Y_2)|^2)dX_2d\Xi_2&\,.
\ea
$$
Hence
$$
\ba
\frac{d \cD_{\hbar,N}}{dt}(t)\le \cD_{\hbar,N}(t)
\\
-\int_{(\bR^d\times\bR^d)^2}\Tr_{\fH_2}(Q_{\hbar,N}^\mathbf{2}(t,X_2,\Xi_2)\cV(t,x_1,x_2)\cdot(\xi_1+i\hbar\grad_{y_1}))dX_2d\Xi_2
\\
+\tfrac12\int_{\bR^d\times\bR^d}\Tr_{\fH}(Q^\mathbf{1}_{\hbar,N}(t,x_1,\xi_1)|\xi_1+i\hbar\grad_{y_1}|^2)dx_1d\xi_1
\\
+L^2\int_{(\bR^d\times\bR^d)^2}\Tr_{\fH_2}(Q^\mathbf{2}_{\hbar,N}(t,X_2,\Xi_2)(|x_1-y_1|^2+|x_2-y_2|^2))dX_2d\Xi_2
\\
=\cD_{\hbar,N}(t)
\\
-\int_{(\bR^d\times\bR^d)^2}\Tr_{\fH_2}(Q_{\hbar,N}^\mathbf{2}(t,X_2,\Xi_2)\cV(t,x_1,x_2)\cdot(\xi_1+i\hbar\grad_{y_1}))dX_2d\Xi_2
\\
+\tfrac12\int_{\bR^d\times\bR^d}\Tr_{\fH}(Q^\mathbf{1}_{\hbar,N}(t,x_1,\xi_1)|\xi_1+i\hbar\grad_{y_1}|^2)dx_1d\xi_1
\\
+2L^2\int_{\bR^d\times\bR^d}\Tr_{\fH}(Q^\mathbf{1}_{\hbar,N}(t,x_1,\xi_1)|x_1-y_1|^2)dx_1d\xi_1\,,
\ea
$$
where the inequality follows from the Lipschitz continuity of $\grad V$ and the equality from the fact that $Q_{\hbar,N}$ is symmetric.

Eventually, we arrive at the inequality
$$
\ba
\frac{d \cD_{\hbar,N}}{dt}(t)\le(1+\max(1,4\Lip(\grad V)^2))\cD_{\hbar,N}(t)&
\\
-\int_{(\bR^d\times\bR^d)^2}\Tr_{\fH_2}(Q_{\hbar,N}^\mathbf{2}(t,X_2,\Xi_2)\cV(t,x_1,x_2)\cdot(\xi_1+i\hbar\grad_{y_1}))dX_2d\Xi_2&\,.
\ea
$$

\subsubsection{Consistency}


The consistency part of the proof is the control of the term
$$
\ba
\int_{(\bR^d\times\bR^d)^2}\Tr_{\fH}(Q_{\hbar,N}^\mathbf{2}(t,X_2,\Xi_2)\cV(t,x_1,x_2)\cdot(\xi_1+i\hbar\grad_{y_1}))dX_2d\Xi_2&
\ea
$$
on the right hand side of the inequality above.

At this point, we undo the symmetry reduction leading to equation (\ref{BBGKY1}), and distribute the interaction of particle $1$ with particle $2$ evenly into interactions of particle $1$ with particles $2,\ldots,N$.  In other words
$$
\ba
\int_{(\bR^d\times\bR^d)^2}\Tr_{\fH_2}(Q_{\hbar,N}^\mathbf{2}(t,X_2,\Xi_2)\cV(t,x_1,x_2)\cdot(\xi_1+i\hbar\grad_{y_1}))dX_2d\Xi_2
\\
=\int_{(\bR^d\times\bR^d)^N}\Tr_{\fH_N}(Q_{\hbar,N}(t,\si\cdot X_N,\si\cdot\Xi_N)\cV(t,x_1,x_j)\cdot(\xi_1+i\hbar\grad_{y_1}))dX_Nd\Xi_N
\\
=\int_{(\bR^d\times\bR^d)^N}\Tr_{\fH_N}(Q_{\hbar,N}(t,X_N,\Xi_N)U_\si\cV(t,x_1,x_j)\cdot(\xi_1+i\hbar\grad_{y_1})U^*_\si)dX_Nd\Xi_N
\\
=\int_{(\bR^d\times\bR^d)^N}\Tr_{\fH_N}(Q_{\hbar,N}(t,X_N,\Xi_N)\cV(t,x_1,x_j)\cdot(\xi_1+i\hbar\grad_{y_1}))dX_Nd\Xi_N
\ea
$$
for all $j=2,\ldots,N$, where $\si$ is the permutation exchanging $2$ and $j$ and leaving all the other indices invariant. 
Therefore
$$
\ba
\int_{(\bR^d\times\bR^d)^2}\Tr_{\fH}(Q_{\hbar,N}^\mathbf{2}(t,X_2,\Xi_2)\cV(t,x_1,x_2)\cdot(\xi_1+i\hbar\grad_{y_1}))dX_2d\Xi_2&
\\
=\int_{(\bR^d\times\bR^d)^N}\Tr_{\fH_N}\left(Q_{\hbar,N}\tfrac1{N-1}\sum_{j=2}^N\cV(t,x_1,x_j)\cdot(\xi_1+i\hbar\grad_{y_1})\right)dX_Nd\Xi_N&\,.
\ea
$$
Applying inequality (\ref{NCCS}) with 
$$
A=\tfrac1{N-1}\sum_{j=2}^N\cV(t,x_1,x_j)\,,\quad B=\xi_1+i\hbar\grad_{y_1}
$$ 
shows that
$$
\ba
\int_{(\bR^d\times\bR^d)^2}\Tr_{\fH_2}(Q_{\hbar,N}^\mathbf{2}(t,X_2,\Xi_2)\cV(t,x_1,x_2)\cdot(\xi_1+i\hbar\grad_{y_1}))dX_2d\Xi_2&
\\
\le\tfrac12\int_{(\bR^d\times\bR^d)^N}\Tr_{\fH_N}\left(Q_{\hbar,N}\left|\tfrac1{N-1}\sum_{j=2}^N\cV(t,x_1,x_j)\right|^2\right)dX_Nd\Xi_N&
\\
+\tfrac12\int_{(\bR^d\times\bR^d)^N}\Tr_{\fH_N}(Q_{\hbar,N}|\xi_1+i\hbar\grad_{y_1}|^2)dX_Nd\Xi_N&\,.
\ea
$$
Since $Q_{\hbar,N}(t,x,\xi)$ acts on the $Y_N$ variables only and $\Tr_{\fH_N}(Q_{\hbar,N}(t))=f(t)^{\otimes N}$ (see Lemma \ref{L-NQCoupl} (a), and the remark thereafter), one has
$$
\ba
\int_{(\bR^d\times\bR^d)^N}\Tr_{\fH_N}\left(Q_{\hbar,N}\left|\tfrac1{N-1}\sum_{j=2}^N\cV(t,x_1,x_j)\right|^2\right)dX_Nd\Xi_N
\\
=\int_{(\bR^d\times\bR^d)^N}f(t)^{\otimes N}(X_N,\Xi_N)\left|\tfrac1{N-1}\sum_{j=2}^N\cV(t,x_1,x_j)\right|^2dX_Nd\Xi_N
\\
=\int_{(\bR^d)^N}\left|\tfrac1{N-1}\sum_{j=2}^N\cV(t,x_1,x_j)\right|^2\rho_f(t)^{\otimes N}(X_N)dX_N&\,,
\ea
$$
where the last equality follows from the fact that the potential $V$ is independent of the momentum variable $\xi$.

This last term is mastered as follows (see Lemma 3.3 in \cite{FGMouPaul} in the case $p=2$ for more details):
$$
\ba
\int_{(\bR^d)^N}\left|\tfrac1{N-1}\sum_{j=2}^N\cV(t,x_1,x_j)\right|^2\rho_f(t)^{\otimes N}(X_N)dX_N
\\
=\frac1{(N-1)^2}\int_{(\bR^d)^2}\sum_{j=2}^N|\cV(t,x_1,x_j)|^2\rho_f(t,x_1)\rho_f(t,x_j)dx_1dx_j
\\
+\frac{N-2}{N-1}\int_{\bR^d}\left(\int_{\bR^d}\cV(t,x_1,x_2)\rho_f(t,x_2)dx_2\right)^2\rho_f(t,x_1)dx_1
\\
\le\frac{2}{N-1}(2\|\grad V\|_{L^\infty})^2\,.
\ea
$$

Finally, we have proved that
$$
\ba
\frac{d \cD_{\hbar,N}}{dt}(t)\le&(1+\max(1,4\Lip(\grad V)^2))\cD_{\hbar,N}(t)
\\
&+\tfrac12\int_{(\bR^d\times\bR^d)^N}\Tr_{\fH_N}(Q_{\hbar,N}|\xi_1+i\hbar\grad_{y_1}|^2)dX_Nd\Xi_N
\\
&+\tfrac12\int_{(\bR^d)^N}\left|\tfrac1{N-1}\sum_{j=2}^N\cV(t,x_1,x_j)\right|^2\rho_f(t)^{\otimes N}(X_N)dX_N
\\
\le&(2+\max(1,4\Lip(\grad V)^2))\cD_{\hbar,N}(t)+\frac{(2\|\grad V\|_{L^\infty})^2}{N-1}\,.
\ea
$$
By Gronwall's inequality
$$
\cD_{\hbar,N}(t)\le \cD_{\hbar,N}(0)e^{\Gamma t}+\frac{(2\|\grad V\|_{L^\infty})^2}{N-1}\frac{e^{\Gamma t}-1}{\Gamma }.
$$

\subsection{Step 5: Conclusion.}\label{concl}


Observe that, for each $n=1,\ldots,N$
$$
Q^\mathbf{n}_{\hbar,N}(t)\in\cC(f(t)^{\otimes n},R^\mathbf{n}_{\hbar,N}(t))\,,\quad\hbox{ for each }t\ge 0\,.
$$
Indeed
$$
\ba
\Tr_{\fH_n}Q^\mathbf{n}_{\hbar,N}(t,X_n,\Xi_n)&
\\
=\Tr_{\fH_n}\int_{(\bR^d\times\bR^d)^{N-n}}[Q_{\hbar,N}(t,X_N,\Xi_N)]^\mathbf{n}dx_{n+1}d\xi_{n+1}\ldots dx_Nd\xi_N&
\\
=\int_{(\bR^d\times\bR^d)^{N-n}}\Tr_{\fH_N}Q_{\hbar,N}(t,X_N,\Xi_N)dx_{n+1}d\xi_{n+1}\ldots dx_Nd\xi_N&
\\
=\int_{(\bR^d\times\bR^d)^{N-n}}f^{\otimes N}(t,X_N,\Xi_N)dx_{n+1}d\xi_{n+1}\ldots dx_Nd\xi_N&
\\
=f^{\otimes n}(t,X_n,\Xi_n)&\,,
\ea
$$
while
$$
\ba
\int_{(\bR^d\times\bR^d)^n}Q^\mathbf{n}_{\hbar,N}(t,X_n,\Xi_n)dX_nd\Xi_n&
\\
=\left[\int_{(\bR^d\times\bR^d)^N}Q_{\hbar,N}(t,X_N,\Xi_N)dX_Nd\Xi_N\right]^\mathbf{n}=R^\mathbf{n}_{\hbar,N}(t)&\,.
\ea
$$
Besides, specializing the symmetry in Lemma \ref{L-NQCoupl} (b) to the case where $\si(k)=k$ for each $k=n+1,\ldots,N$ shows that 
\be\lb{CouplMargi}
Q^\mathbf{n}_{\hbar,N}(t)\in\cC^s(f(t)^{\otimes n},R^\mathbf{n}_{\hbar,N}(t))\,,\quad\hbox{ for each }t\ge 0\,.
\ee

By definition of the pseudo-distance $E_\hbar$, one has
$$
\frac1nE_\hbar(f(t)^{\otimes n},R_{\hbar, N}^\mathbf{n}(t))\le \cD_{\hbar,N}(0)e^{\Gamma t}+\frac{(2\|\grad V\|_{L^\infty})^2}{N-1}\frac{e^{\Gamma t}-1}{\Gamma }
$$
for each $n=1,\ldots,N$ and each $t\ge 0$. Minimizing the right hand side of the inequality above as the initial $Q^{in}_{\hbar,N}$ runs through $\cC^s((f^{in})^{\otimes N},R_{\hbar,N}^{in})$ , we arrive at the inequality
$$
\frac1nE_\hbar(f(t)^{\otimes n},R_{\hbar, N}^\mathbf{n}(t))\le\frac1NE_\hbar((f^{in})^{\otimes N},R^{in}_{\hbar, N})e^{\Gamma t}+\frac{(2\|\grad V\|_{L^\infty})^2}{N-1}\frac{e^{\Gamma t}-1}{\Gamma }\,.
$$
(Indeed
$$
\ba
\int_{(\bR^d\times\bR^d)^N}\Tr_{\fH_N}(c_{\hbar,j}(x_j,\xi_j)Q^{in}_{\hbar,N}(X_N,\Xi_N))dX_Nd\Xi_N
\\
=\int_{(\bR^d\times\bR^d)^N}\Tr_{\fH_N}(c_{\hbar,1}(x_1,\xi_1)U_\si Q^{in}_{\hbar,N}(\si\cdot X_N,\si\cdot\Xi_N)U^*_\si)dX_Nd\Xi_N
\ea
$$
if $\si$ is the transposition exchanging $1$ and $j$, and therefore
\be\lb{MinSym}
\ba
\int_{(\bR^d\times\bR^d)^N}\frac1N\sum_{j=1}^N\Tr_{\fH_N}(c_{\hbar,j}(x_j,\xi_j)Q^{in}_{\hbar,N}(X_N,\Xi_N))dX_Nd\Xi_N
\\
=
\int_{(\bR^d\times\bR^d)^N}\frac1N\sum_{j=1}^N\Tr_{\fH_N}(c_{\hbar,j}(x_j,\xi_j)Q^{in,sym}_{\hbar,N}(X_N,\Xi_N))dX_Nd\Xi_N
\ea
\ee
where
$$
Q^{in,sym}_{\hbar,N}(X_N,\Xi_N)=\frac1{N!}\sum_{\si\in\fS_N}U_\si Q^{in}_{\hbar,N}(\si\cdot X_N,\si\cdot\Xi_N)U^*_\si\,,
$$
so that there is no loss of generality in assuming that $Q^{in}_{\hbar,N}$ satisfies the symmetry (\ref{NSym}).)

By statement (2) in Theorem \ref{T-PtyE}
$$
\ba
\frac1n\MKd(f(t)^{\otimes n},\widetilde W_\hbar[R_{\hbar, N}^\mathbf{n}(t)])^2
\\
\le\frac1NE_\hbar((f^{in})^{\otimes N},R^{in}_{\hbar, N})e^{\Gamma t}+\frac{(2\|\grad V\|_{L^\infty})^2}{N-1}\frac{e^{\Gamma t}-1}{\Gamma }+\tfrac12d\hbar&\,.
\ea
$$
If $R_{\hbar,N}^{in}$ is a T\"oplitz operator, more precisely if $R_{\hbar,N}^{in}=\Op^T_\hbar((2\pi\hbar)^{dN}\mu^{in}_N)$ for some symmetric Borel probablity measure $\mu^{in}_N$ on $(\bR^d)^N$, one has
$$
\ba
\frac1n\MKd(f(t)^{\otimes n},\widetilde W_\hbar[R_{\hbar, N}^\mathbf{n}(t)])^2
\\
\le\left(\frac1N\MKd((f^{in})^{\otimes N},\mu^{in}_N)^2+\tfrac12d\hbar\right)e^{\Gamma t}+\frac{(2\|\grad V\|_{L^\infty})^2}{N-1}\frac{e^{\Gamma t}-1}{\Gamma }+\tfrac12d\hbar&\,.
\ea
$$

\begin{Rmk}
The argument in Step 4 can be summarized as follows: the first equation in the BBGKY hierarchy at the level of couplings (i.e. equation (\ref{BBGKY1})) is used in the stability part of the convergence analysis, while the last equation
in the BBGKY hierarchy, or equivalently the $N$-particle equation (\ref{EqNCoupl}) at the level of couplings, is used in the consistency part of the proof. None of the intermediate equations in the BBGKY hierarchy is used in the proof.
In view of this remark, it is interesting to compare the method used in the proof of Theorem \ref{T-NSV} with the abstract argument for hierarchies outlined in \cite{BEGMY}, which is based on a Cauchy-Kovalevska argument.
\end{Rmk}


\section{Proof of Theorem \ref{T-SL}}


The proof of  Theorem \ref{T-SL} is similar to the proofs of Theorems \ref{T-HV} and \ref{T-NSV}. We only sketch the argument, and insist on the differences with the proofs of Theorems \ref{T-HV} and \ref{T-NSV}.

Let $Q^{in}_{\hbar,N}\in\cC^s(F^{in}_N,R^{in}_{\hbar,N})$, and let $(t,X_N,\Xi_N)\mapsto Q_{\hbar,N}(t,X_N,\Xi_N)$ be the solution of the Cauchy problem
\be\lb{CouplNSemiC}
\left\{
\ba
{}&\d_tQ_{\hbar,N}+\{\bH_N,Q_{\hbar,N}\}_N+\frac{i}{\hbar}[\cH_{\hbar,N},Q_{\hbar,N}]_N=0\,,
\\
&Q_{\hbar,N}\rstr_{t=0}=Q^{in}_{\hbar,N}\,,
\ea
\right.
\ee
where the classical Hamiltonian $\bH_N$ is defined in (\ref{CHamN}) and the quantum Hamiltonian $\cH_{\hbar,N}$ in (\ref{QHamN}). By the same argument as in Lemma \ref{L-NQCoupl}, we see that
$$
Q_{\hbar,N}(t)\in\cC^s(F_N(t),R_{\hbar,N}(t))\,,\quad\hbox{ for each }t\ge 0\,.
$$

Set
$$
\ba
\cD_{\hbar,N}(t)=\frac1N\int_{(\bR^d\times\bR^d)^N}\Tr_{\fH_N}\left(Q_{\hbar,N}(t,X_N,\Xi_N)\sum_{j=1}^Nc_{\hbar,j}\right)dX_Nd\Xi_N
\\
=\int_{(\bR^d\times\bR^d)^N}\Tr_{\fH_N}\left(Q^\mathbf{1}_{\hbar,N}(t,x_1,\xi_1)c_{\hbar,1}\right)dx_1d\xi_1\,.
\ea
$$
as in Lemma \ref{L-DMargi} with $n=1$.

Multiplying each side of (\ref{CouplNSemiC}) by $c_{\hbar,1}$, taking the trace and integrating in $(x_1,\xi_1)$, we get
$$
\ba
\dot{\cD}_{\hbar,N}(t)=\int_{\bR^d\times\bR^d}\Tr_{\fH}(Q^\mathbf{1}_{\hbar,N}(t,x_1,\xi_1)\{\tfrac12\xi_1,c_{\hbar,1}(x_1,\xi_1)\})dx_1d\xi_1
\\
+\int_{\bR^d\times\bR^d}\Tr_{\fH}(Q^\mathbf{1}_{\hbar,N}(t,x_1,\xi_1)[-\tfrac{i}2\hbar\Dlt_{y_1},c_{\hbar,1}(x_1,\xi_1)])dx_1d\xi_1
\\
+\int_{(\bR^d\times\bR^d)^2}\Tr_{\fH_2}(Q^\mathbf{2}_{\hbar,N}(t,X_2,\Xi_2)\{\tfrac{N-1}NV(x_1-x_2),c_{\hbar,1}(x_1,\xi_1)\})dX_2d\Xi_2
\\
+\int_{\bR^d\times\bR^d}\frac{i}{\hbar}\Tr_{\fH_2}(Q^\mathbf{2}_{\hbar,N}(t,X_2,\Xi_2)[\tfrac{N-1}NV(y_1-y_2),c_{\hbar,1}(x_1,\xi_1)])dX_2d\Xi_2\,.
\ea
$$

One has
$$
\{\tfrac12\xi_1,c_{\hbar,1}(x_1,\xi_1)\}=\xi_1\cdot\grad_{x_1}c_{\hbar,1}(x_1,\xi_1)=\xi_1\cdot(x_1-y_1)\,,
$$
while
$$
[-\tfrac{i}2\hbar\Dlt_{y_1},c_{\hbar,1}(x_1,\xi_1)]=-\tfrac12i\hbar\grad_{y_1}\cdot(y_1-x_1)-\tfrac12i\hbar(y_1-x_1)\cdot\grad_{y_1}\,.
$$

Likewise
$$
\{\tfrac{N-1}NV(x_1-x_2),c_{\hbar,1}(x_1,\xi_1)\}=-\tfrac{N-1}N\grad V(x_1-x_2)\cdot(\xi_1+i\hbar\grad_{y_1})\,,
$$
while
$$
\ba
\frac{i}{\hbar}[\tfrac{N-1}NV(y_1-y_2),c_{\hbar,1}(x_1,\xi_1)]=-\frac{i}{\hbar}\tfrac{N-1}N[c_{\hbar,1}(x_1,\xi_1),V(y_1-y_2)]
\\
=-\frac{i}{\hbar}\tfrac{N-1}N\left(\tfrac12(\xi_1+i\hbar\grad_{y_1})\cdot i\hbar\grad V(y_1-y_2)+\tfrac12i\hbar\grad V(y_1-y_2)\cdot (\xi_1+i\hbar\grad_{y_1})\right)
\\
=\tfrac{N-1}N\left(\tfrac12(\xi_1+i\hbar\grad_{y_1})\cdot\grad V(y_1-y_2)+\tfrac12\grad V(y_1-y_2)\cdot (\xi_1+i\hbar\grad_{y_1})\right)\,.
\ea
$$
Thus
$$
\ba
\dot{\cD}_{\hbar,N}(t)=\int_{\bR^d\times\bR^d}\Tr_{\fH}(Q^\mathbf{1}_{\hbar,N}(t,x_1,\xi_1)\tfrac12(\xi_1+i\hbar\grad_{y_1})\cdot(x_1-y_1))dx_1d\xi_1
\\
+\int_{\bR^d\times\bR^d}\Tr_{\fH}(Q^\mathbf{1}_{\hbar,N}(t,x_1,\xi_1)\tfrac12(x_1-y_1)\cdot(\xi_1+i\hbar\grad_{y_1}))dx_1d\xi_1
\\
-\frac{N-1}{2N}\int_{(\bR^d\times\bR^d)^2}\Tr_{\fH_2}(Q^\mathbf{2}_{\hbar,N}(t,X_2,\Xi_2)(\xi_1+i\hbar\grad_{y_1})\cdot\mathcal W(X_2,Y_2))dX_2d\Xi_2
\\
-\frac{N-1}{2N}\int_{(\bR^d\times\bR^d)^2}\Tr_{\fH_2}(Q^\mathbf{2}_{\hbar,N}(t,X_2,\Xi_2)\mathcal W(X_2,Y_2)\cdot(\xi_1+i\hbar\grad_{y_1}))dX_2d\Xi_2
\\
\le \cD_{\hbar,N}(t)+\frac{N-1}{2N}\int_{(\bR^d\times\bR^d)^2}\Tr_{\fH_2}(Q^\mathbf{2}_{\hbar,N}(t,X_2,\Xi_2)|\xi_1+i\hbar\grad_{y_1}|^2)dX_2d\Xi_2
\\
+\frac{N-1}{2N}\int_{(\bR^d\times\bR^d)^2}\Tr_{\fH_2}(Q^\mathbf{2}_{\hbar,N}(t,X_2,\Xi_2)|\mathcal W(X_2,Y_2)|^2)dX_2d\Xi_2&\,,
\ea
$$
with
$$
\mathcal W(X_2,Y_2):=(\grad V(x_1-x_2)-\grad V(y_1-y_2))\,.
$$
Denoting again  $L:=\Lip(\grad V)$, one has
$$
|\mathcal W(X_2,Y_2)|\le L(|x_1-y_1|+|x_2-y_2|)
$$
and therefore
$$
\ba
\dot{\cD}_{\hbar,N}(t)\le \cD_{\hbar,N}(t)+\tfrac12\int_{(\bR^d\times\bR^d)^2}\Tr_{\fH_2}(Q^\mathbf{2}_{\hbar,N}(t,X_2,\Xi_2)|\xi_1+i\hbar\grad_{y_1}|^2)dX_2d\Xi_2
\\
+\tfrac12L^2\int_{(\bR^d\times\bR^d)^2}\Tr_{\fH_2}(Q^\mathbf{2}_{\hbar,N}(t,X_2,\Xi_2)(|x_1-y_1|+|x_2-y_2|)^2)dX_2d\Xi_2
\\
\le \cD_{\hbar,N}(t)+\tfrac12\int_{\bR^d\times\bR^d}\Tr_{\fH}(Q^\mathbf{1}_{\hbar,N}(t,x_1,\xi_1)|\xi_1+i\hbar\grad_{y_1}|^2)dx_1d\xi_1
\\
+2L^2\int_{(\bR^d\times\bR^d)^2}\Tr_\fH(Q^\mathbf{1}_{\hbar,N}(t,x_1,\xi_1)|x_1-y_1|^2)dx_1d\xi_1
\\
\le(1+\max(1,4L^2))\cD_{\hbar,N}(t)&\,.
\ea
$$
Hence, for all $t\ge 0$ and all $n=1,\ldots,N$, one has
$$
\frac1nE_\hbar(F^\mathbf{n}_N(t),R^\mathbf{n}_{\hbar,N}(t))^2\le \cD_{\hbar,N}(t)\le \cD_{\hbar,N}(0)e^{\L t}\,.
$$
Minimizing over the initial coupling $Q^{in}_{\hbar,N}\in\cC^s(F^{in}_N,R^{in}_{\hbar,N})$, and arguing as in (\ref{MinSym}), one finds that
$$
\inf_{Q^{in}_{\hbar,N}\in\cC^s(F^{in}_N,R^{in}_{\hbar,N})}\cD_{\hbar,N}(0)^2=\frac1NE_\hbar(F^{in}_N,R^{in}_{\hbar,N})^2\,.
$$
Thus, for all $t\ge 0$ and all $n=1,\ldots,N$, 
$$
\frac1nE_\hbar(F^\mathbf{n}_N(t),R^\mathbf{n}_{\hbar,N}(t))^2\le \frac1NE_\hbar(F^{in}_N,R^{in}_{\hbar,N})^2 e^{\L t}\,.
$$
We conclude following the end of the proof of Theorem \ref{T-NSV}, Section \ref{concl}.

\vskip 1cm
\noindent
\textbf{Acknowledgments:} This work has been partially carried out thanks to the support of the A*MIDEX project (n$^o$ ANR-11-IDEX-0001-02) funded by the ``Investissements d'Avenir" French Government program, managed by the French 
National Research Agency (ANR). T.P. thanks also the Dipartimento di Matematica, Sapienza Universit\`a di Roma, for its kind hospitality during the completion of this work. Finally we thank both referees for their suggestions which  helped us
improving the presentation of this work.

\smallskip
\noindent
\textbf{Conflict of interest:} The authors declare that they have no conflict of interest.



\begin{thebibliography}{99}


\bibitem{APPP1}
A. Athanassoulis, T. Paul, F. Pezzotti, M. Pulvirenti: 
\textit{Strong Semiclassical Approximation of Wigner Functions for the Hartree Dynamics}, 
Rend. Lincei: Mat. e Appl., \textbf{22} (2011), 525--552.
  
\bibitem{APPP2}  
A. Athanassoulis, T. Paul, F. Pezzotti, M. Pulvirenti: 
\textit{Semiclassical Propagation of Coherent States for the Hartree Equation},  
Ann. H. Poincar\'e, \textbf{12} (2011), 1613--1634.

\bibitem{BEGMY}
C. Bardos, L. Erd\"os, F. Golse, N. Mauser, H.-T. Yau:
\textit{Derivation of the Schr\"odinger-Poisson equation from the quantum $N$-body problem},
C. R. Acad. Sci. Paris, S\'er. I  \textbf{334} (2002), 515--520.

\bibitem{BGM}
C. Bardos, F. Golse, N. Mauser:
\textit{Weak Coupling Limit of the $N$-Particle Schr\"odinger Equation},
Meth. Applications Anal. \textbf{7} (2000), 275--294.

\bibitem{BPSS}
N. Benedikter, M. Porta, C. Saffirio B. Schlein: 
\textit{From the Hartree dynamics to the Vlasov equation}, 
Arch. Rational Mech. Anal.  \textbf{221} (2016), 273--334.

\bibitem{BeSh}
F.A. Berezin, M.A. Shubin:
``The Schr\"odinger equation''. Translated from the 1983 Russian edition.
Mathematics and its Applications (Soviet Series), 66. 
Kluwer Academic Publishers Group, Dordrecht, 1991.

\bibitem{BoveDPF}
A. Bove, G. DaPrato, G. Fano:
\textit{An existence proof for the Hartree-Fock time-dependent problem with bounded two-body interaction}, 
Commun. Math. Phys. \textbf{37} (1974), 183--191.

\bibitem{Dobru}
R. Dobrushin: 
\textit{Vlasov equations}, 
Funct. Anal. Appl. \textbf{13} (1979), 115--123.

\bibitem{FGMouPaul}
F. Golse, C. Mouhot, T. Paul:
\textit{On the Mean-Field and Classical Limits of Quantum Mechanics},
Commun. Math. Phys. \textbf{343} (2016), 165--205.

\bibitem{GMP}
S. Graffi, A. Martinez, M. Pulvirenti:
\textit{Mean-field approximation of quantum systems and classical limit},
Math. Models Methods Appl. Sci. \textbf{13} (2003), 59--73.

\bibitem{HaurayJabin}
M. Hauray, P.-E. Jabin:
\textit{Particle approximations of Vlasov equations with singular forces},
Ann. Sci. Ecol. Norm. Sup. \textbf{48} (2015), 891--940.

\bibitem{Kato}
T. Kato:
``Perturbation Theory for Linear Operators''. Reprint of the 1980 edition.
Springer-Verlag, Berlin, Heidelberg, 1995.

\bibitem{LaxRicht}
P.D. Lax, R. Richtmeyer:
\textit{Survey of the Stability of Linear Finite Difference Equations},
Comm. on Pure Appl. Math. \textbf{9} (1956), 267--293.

\bibitem{Laza}
D. Lazarovici:
\textit{The Vlasov-Poisson dynamics as the mean-field limit of rigid charges},
preprint {\tt arXiv:1502.07047}.

\bibitem{PicklLaza}
D. Lazarovici, P. Pickl:
\textit{A mean-field limit for the Vlasov-Poisson system},
preprint {\tt arXiv:1502.04608}.

\bibitem{Lerner}
N. Lerner: 
\textit{Some facts about the Wick calculus}, 
in ``Pseudodifferential operators'', 135--174, 
Lecture Notes in Math., 1949, Springer, Berlin, 2008.

\bibitem{LernerBook}
N. Lerner: 
``Metrics on the phase space and non-selfadjoint pseudo-differential operators'',
Birkh\"auser Verlag, Basel, 2010.

\bibitem{LionsPaul}
P.-L. Lions, T. Paul:
\textit{Sur les mesures de Wigner}, 
Rev. Mat. Iberoamericana \textbf{9} (1993), 553--618.

\bibitem{Loeper}
G. Loeper: 
\textit{Uniqueness of the solution to the Vlasov-Poisson system with bounded density},
J. Math. Pures Appl. \textbf{86} (2006) 68--79.

\bibitem{NarnhoSewell}
H. Narnhofer, G. Sewell:
\textit{Vlasov hydrodynamics of a quantum mechanical model}, 
Commun. Math. Phys. \textbf{79} (1981), 9--24.

\bibitem{PP}
F. Pezzoti, M Pulvirenti: 
\textit{Mean-Field Limit and Semiclassical Expansion of a Quantum Particle System},  
Ann. Henri Poincar\'e \textbf{10} (2009) 145-187.

\bibitem{BSimon}
B. Simon:
\textit{The Classical Limit of Quantum Partition Functions},
Commun. Math. Phys. \textbf{71} (1980), 247--276.

\bibitem{Spohn}
H. Spohn: 
\textit{On the Vlasov hierarchy}, 
Math. Meth. in the Appl. Sci. \textbf{3} (1981), 445--455.

\bibitem{VillaniAMS}
C. Villani: 
``Topics in Optimal Transportation'', 
Amer. Math. Soc., Providence (RI), 2003.

\bibitem{VillaniTOT}
C. Villani: 
``Optimal Transport. Old and New'', 
Springer-Verlag, Berlin, Heidelberg, 2009.

\end{thebibliography}
\end{document}